\documentclass[10pt]{amsart}
\pdfoutput=1
\theoremstyle{plain}
\usepackage{amsmath, amsfonts, amsthm, amssymb,amscd,bbold}
\usepackage{graphics}
\usepackage{color}
\usepackage{epsfig}
\usepackage{tikz}
\usepackage{tikz-cd}

\usepackage{bbm}
\usepackage{multirow}

\usepackage{enumitem}
\usepackage{chngcntr}
\usepackage{float}
\usepackage[export]{adjustbox}
\usepackage{stackengine}

\usepackage[normalem]{ulem}
\usepackage{tcolorbox}
\usepackage{ stmaryrd }
\usepackage{array}
\usepackage{calc}
\usepackage{caption}
\usepackage{subcaption}
\usepackage[export]{adjustbox}

\usetikzlibrary{arrows}
\usepackage[all]{xy}
\usepackage{todonotes}
\theoremstyle{plain}
\newtheorem{theorem}{Theorem}
\newtheorem{thm}{Theorem}
\newtheorem{lemma}{Lemma}

\newtheorem{corollary}{Corollary}
\newtheorem{cor}{Corollary}

\theoremstyle{definition}

\newtheorem{definition}{Definition}

\newtheorem*{remark}{Remark}
\newtheorem{example}{Example}

\newcommand{\FF}{\ensuremath{\mathbb{F}}}

\DeclareMathOperator{\Id}{Id}

\newcommand{\gl}{\mathfrak{gl}}

\DeclareMathOperator{\Mod}{-mod}

\counterwithin{theorem}{section}

\author{Champ Davis}
\address{University of Oregon; Department of Mathematics; 108C University Hall; Eugene, OR 97403}
\email{champd@uoregon.edu}

\title{Restriction of Scalars for $L_\infty$-modules}

\begin{document}
\bibliographystyle{plain}


\begin{abstract}
Let $I: L' \to L$ be a morphism of $L_\infty$-algebras.  The goal of this paper is to describe  restriction of scalars in the setting of $L_\infty$-modules and prove that it defines a functor $I^*: L\text{-mod} \to L'\text{-mod}$.  A more abstract approach to this problem was recently given by Kraft-Schnitzer.  In a subsequent paper, this result is applied to show that there is a well-defined $L_\infty$-module structure on the sutured annular Khovanov homology of a link in a thickened annulus.  

\end{abstract}

\maketitle

\section{Introduction}\label{sec:Intro}

The study of $L_\infty$-algebras, also known as strong homotopy Lie algebras or sh-Lie algebras, can be traced back to rational homotopy theory and the deformations of algebraic structures, where they first appeared in the form of Lie-Massy operations \cite{A,R, JS2}.  Early applications centered around the Quillen spectral sequence and rational Whitehead products, and there has been continued interest in higher order Whitehead products recently; see \cite{BBM}.  There has also been much interest in $L_\infty$-algebras in physics, where Lie algebras and their representations play a major role.  In particular, $L_\infty$-algebra structures have appeared in work on higher spin particles \cite{BBvD}, as well as in closed string theory \cite{WZ1,Z1}.  Stasheff gives a nice overview in a recent survey article \cite{JS1}.

Attention has also been given to modules over $L_\infty$-algebras.  The notion of an $L_\infty$-module was introduced in \cite{LM1}, in which the correspondence between Lie algebra representations and Lie modules was generalized to the $L_\infty$ setting.  Moreover, homomorphisms between $L_\infty$-modules were developed in \cite{A1}. 

While it is possible to give a complex an $L_\infty$ structure by writing down explicit formulas, another  option is to use homological perturbation theory to transfer an existing $L_\infty$ structure from a different complex.  Information on how to do so can be found in \cite{JH1,H-S,GLS}, where this idea is referred to as the homological perturbation lemma, though sometimes it is referred to as the homotopy transfer theorem, as in \cite{L-V,Man}.  An approach using operads was given in \cite{Berg}, where explicit formulas are written down for the $A_\infty$ case.  Explicit formulas for the $L_\infty$ case can be found in \cite{JMMF}.  

Much of the literature deals with the transfer of $L_\infty$-algebra structures; however, given a map between $L_\infty$-algebras, it is natural to want to use this map to relate their respective categories of modules.  In this paper, we give one explicit formula to do so, giving a proof of the following:

\begin{thm}
	Suppose $L,L'$ are $L_\infty$-algebras over $\FF_2$ and $I: L' \to L$ is a map of $L_\infty$-algebras.  Then there is an induced functor $I^*: L\Mod \to L'\Mod$, called restriction of scalars.
\end{thm}

Given an $L_\infty$-module homomorphism $f: M\to N$, our definition will satisfy $(I^*f)_1=f_1$.  It follows that $I^*$ preserves quasi-isomorphisms; that is, if $M$ and $N$ are quasi-isomorphic, then so too are $I^*M$ and $I^*N$.  We also observe that this generalizes the analagous result in the Lie algebra setting:

\begin{cor}
	If $L$ and $L'$ are Lie algebras, and $\phi: L'\to L$ is a Lie algebra homomorphism, $\phi^*$ is the usual restriction of scalars for Lie algebra representations.
\end{cor}

One reason for interest in this result is a particular application in knot theory.  In \cite{GLW}, it was shown that the sutured annular Khovanov homology of a knot can be given the structure of a Lie algebra representation.  In subsequent work, we will use the formulas from this paper to upgrade this structure to an $L_\infty$-module.

Because $L_\infty$ modules are defined in the graded setting, keeping track of signs requires a great deal of care.  We will ignore signs and work over $\FF_2$.  As mentioned in \cite{A1}, $A_\infty$-modules and maps between them can be reinterpreted in terms of differential comodules.  The analagous reformulation in the $L_\infty$ case is less-understood, but perhaps could facilitate the recording of signs.  Moreover, while this paper was in preparation, Kraft-Schnitzer gave a more abstract approach to the restriction of scalars operation in \cite{KS}.  We present an alternative interpretation, and we emphasize that the explicit formulas developed here are of particular interest for our applications.  On the other hand, \cite{KS} might serve as a guide for how to deal with signs in the future.

The outline of the paper is as follows.   In section 2, we review the definition of an $L_\infty$-algebra and explain morphisms between them.  In section 3, we provide a similar exposition for $L_\infty$-modules, and we also describe how to compose morphisms between $L_\infty$-modules.  In section 4, we describe $I^*$, the restriction of scalars functor.  We define $I^*$ on objects and morphisms, and then we prove that it is functorial.  The appendix includes supplementary graphics for the proofs presented in the paper, which contain somewhat complicated formulas.

\section{$L_\infty$-algebras and $L_\infty$-modules}

	In this section, we review $L_\infty$-algebras and explain morphisms between them.  We start by introducting some notation that we will use throughout the paper.

	\begin{definition}
	 	Let $\sigma \in S^n$ and $x_i$ be elements of a set $X$.  Define the map $\sigma^{\bullet}: X^n \to X^n$ by $\sigma^{\bullet}(x_1, x_2, \ldots, x_n) = (x_{\sigma(1)}, \ldots, x_{\sigma(n)})$.  Note that this induces a map on the $n$-fold tensor product $\sigma^{\bullet}: X^{\otimes n} \to X^{\otimes n}$ that sends an element $x_1 \otimes \cdots \otimes x_n$ to $x_{\sigma(1)} \otimes \cdots \otimes x_{\sigma(n)}$.
	 \end{definition} 

	\begin{definition}
	 	Fix non-negative integers $i_1, i_2, \ldots, i_n$, with $i_1 + i_2 + \cdots + i_r = n$.  A permutation $\sigma\in S_n$ is an \textbf{$(i_1, i_2, \ldots, i_r)$-unshuffle} if 
	 	\begin{align*}
	 		&\sigma(1) < \cdots < \sigma(i_1) \\ 
	 		&\sigma(i_1+1) < \cdots < \sigma(i_1+i_2) \\ 
	 		&\quad \vdots \\ 
	 		&\sigma(i_1+\cdots+i_{r-1}+1) < \cdots < \sigma(i_1+\cdots+i_r)
	 	\end{align*}
	 	We will denote the set of $(i_1, i_2, \ldots, i_r)$-unshuffles in $S_n$ by $S(i_1, \ldots, i_r)$.  We will also denote by $S'(i_1, \ldots, i_r)$ the set of $(i_1, i_2, \ldots, i_r)$-unshuffles $\sigma$ in $S_n$ satisfying $i_1 \leq i_2 \leq \cdots \leq i_r$ and $\sigma(i_1 + \cdots + i_{l-1} + 1) < \sigma(i_1 + \cdots + i_l+1)$ if $i_l = i_{l+1}$.  This second condition on $\sigma$ says that the order is preserved when comparing the first elements of blocks of the same size.  Indeed, if $\sigma$ is a $(1,2,2,3)$-unshuffle in $S'_8$, then $i_2 = i_3 = 2$, so the order must be preserved when comparing the first element of the $i_2$ block to the first element of the $i_3$ block.
  
	 \end{definition}

	 \begin{example} Figure \ref{fig:A} is an example of a $(1,1,2,3)$-unshuffle in $S_7$.  That is, $\sigma=(124653)(7)$, and we have drawn a picture describing $\sigma^{\bullet}$.  That is, $x_{\sigma(1)} = x_2$, $x_{\sigma(2)}=x_4$, and so on.  The picture describes how $\sigma^\bullet$ permutes $x_1, \ldots, x_7$.
	 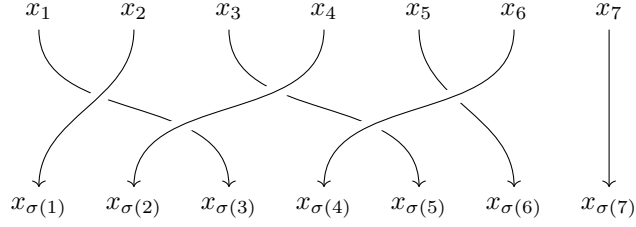
\begin{figure}[h!]
	 	\centering
		\[
		\begin{tikzcd}[column sep=tiny, row sep=6em]
		x_1 \arrow[drr, out=-90, in=90] & x_2 \arrow[dl, out=-90, in=90, crossing over] & x_3 \arrow[drr, out=-90, in=90] & x_4 \arrow[dll, out=-90, in=90, crossing over] & x_5 \arrow[dr, out=-90, in=90] & x_6 \arrow[dll, out=-90, in=90, crossing over] & x_7 \arrow[d] \\ 
		x_{\sigma(1)} & x_{\sigma(2)} & x_{\sigma(3)} & x_{\sigma(4)} & x_{\sigma(5)} & x_{\sigma(6)} & x_{\sigma(7)}
		\end{tikzcd}
		\]
		\caption{A depiction of an $(1,1,2,3)$-unshuffle in $S_7$.  Here $\sigma=(124653)(7)$, and $\sigma^\bullet(x_1, x_2,x_3,x_4,x_5,x_6, x_7) = (x_2,x_4,x_1,x_6,x_3,x_5,x_7)$.}
		\label{fig:A}
		 \end{figure}

	\vspace{1em}

	\noindent In words, a $(1,1,2,3)$-unshuffle places the numbers 1 through 7 into boxes of size 1,1,2, and 3, where the order is preserved in each box.  In this example, the resulting boxes would be $(2), (4), (1,6),$ and $(3,5,7)$.
	\end{example}

	\begin{example}
		A special case of the above definition is if we only have two numbers in our partition of $n$.  In particular, $\sigma\in S_n$ is a \textbf{$(p,n-p)$-unshuffle} if $\sigma(k)<\sigma(k+1)$ whenever $k\neq p$.  In words, this permutation will place the numbers 1 through $n$ into two boxes, where order is preserved in each.  For brevity, we will sometimes refer to a $(p, n-p)$-unshuffle as a $p$-unshuffle if $n$ is clear.
	\end{example}



	\begin{example} In $S_4$, if we use the notation $xyzw$ to denote the permutation
		$\left(\begin{smallmatrix}
			1 & 2 & 3 & 4 \\ 
			x & y & z & w
		\end{smallmatrix}\right)$, 
		then we can write down the 1, 2, and 3-unshuffles:

	$$
	\begin{tabular}{ll}
		1-unshuffles: & 1234, 2134, 3124, 4123 \\ 
		2-unshuffles: & 1234, 1324, 1423, 2314, 2413, 3412 \\ 
		3-unshuffles: & 1234, 1243, 1342, 2341 \\ 
	\end{tabular}
	$$

	\end{example}

We can now state the definition of an $L_\infty$-algebra.  The general definition involves signs, but we are working over $\FF_2$ throughout this paper.

\begin{definition}
	Let $V$ be a graded vector space.  An $L_\infty$-algebra structure on $V$ is a collection of (skew)-symmetric multilinear maps $\{l_k: V^{\otimes k} \to V\}$ of degree $k-2$.  That is,
	$$l_k(x_{\sigma(1)}, x_{\sigma(2)}, \ldots, x_{\sigma(k)}) = l_k(x_1, x_2, \ldots, x_k)$$
	for all $\sigma\in S_k$ and $x_i\in V$.  Moreover, these maps must satisfy the generalized Jacobi identity:
	$$\sum_{i+j=n+1}\sum_{\sigma \in S(i,n-i)} l_j(l_i(x_{\sigma(1)}, \ldots, x_{\sigma(i)}), x_{\sigma(i+1)}, \ldots, x_{\sigma(n)}) = 0$$
	Here, $i\geq 1$, $j\geq 1$, $n\geq 1$, and the inner summation is taken over all $(i,n-i)$-unshuffles.
\end{definition}
\begin{remark}
	We could have also written the (skew)-symmetry condition as $l_k\circ \sigma^{\bullet} = l_k$ for $\sigma \in S_k$.
\end{remark}
\begin{remark}
	Another way to write the generalized Jacobi indentity is by using the notation 
	$$\sum_{i+j=n+1}\sum_{\sigma} l_j \circ (l_i \otimes \Id) \circ \sigma^{\bullet} = 0$$
\end{remark}

\begin{remark} Figure \ref{fig:B} is a depiction of the generalized Jacobi identity.
	\begin{figure}[H]
		\centering
		\scalebox{1}{\input{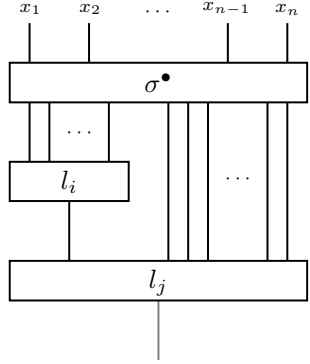}}
    	\caption{A graphical depiction of the generalized Jacobi identity.  This should be interpreted as the sum of all compositions $l_j\circ (l_i \otimes \Id) \circ \sigma^\bullet$, applied to the input $x_1\otimes \cdots \otimes x_n$.  That is, this picture represents $\sum_{i+j=n+1}\sum_{\sigma} l_j \circ (l_i \otimes \Id) \circ \sigma^{\bullet}(x_1\otimes \cdots \otimes x_n) = 0$.} 
    	\label{fig:B}
	\end{figure}
\end{remark}

\begin{remark}
	This definition assumes that our $L_\infty$-algebra is a chain complex.  If instead it is a cochain complex, we require each $l_k$ to have degree $2-k$.
\end{remark}



	

	\begin{definition}
		Let $(L,l_i)$ and $(L', l_i')$ be $L_\infty$ algebras.  An $L_\infty$-algebra homomorphism from $L$ to $L'$ is a collection $\{f_n: L^{\otimes n} \to L'\}$ of (skew)-symmetric multilinear maps of degree $n-1$ such that 
		\begin{align*}
			\sum_{j+k = n+1} \sum_{\sigma\in S(k,n-k)} f_j \circ (l_k \otimes \Id) \circ \sigma^{\bullet}  = \sum_{\substack{\tau \in S'(i_1,\ldots,i_r) \\  i_1+\ldots+i_r=n}} l_r' \circ (f_{i_1} \otimes \cdots \otimes f_{i_r})\circ \tau^{\bullet}
		\end{align*}

	\end{definition}

	\begin{example}  The $n=2$ morphism relation says that
		\begin{equation*}
		\begin{split}
			f_1(l_2(x_1,x_2))+f_2(l_1(x_1),x_2)+f_2(l_1(x_2),x_1) 
			&= l_1'(f_2(x_1,x_2))+l_2'(f_1(x_1),f_1(x_2))
					\end{split}
		\end{equation*}
		When $(L, l_i)$ and $(L',l_i')$ are $L_\infty$-algebras consisting of elements in degree 0 only, the $n=2$ morphism relation simplifies to
		$f_1(l_2(x_1,x_2)) + l_2'(f_1(x_1),f_1(x_2)) = 0$,
		which is just a Lie algebra homomorphism:
		$\phi([x_1, x_2]) = [\phi(x_1),\phi(x_2)]$.
		
	\end{example}

\begin{definition}
	Let $(L,l_k)$ be an $L_\infty$-algebra.  An \textbf{$L_\infty$-module} over $L$ is a graded vector space $M$, together with a collection of skew-symmetric multilinear maps $\{k_n: L^{\otimes n-1} \otimes M \to M \mid 1\leq n<\infty\}$ of degree $n-2$ such that the following identity holds:
	\begin{align*}
			&\sum_{\substack{p+q=n+1 \\ p<n}}\sum_{\sigma(n)=n} k_q\circ (l_p \otimes \Id) \circ \sigma^{\bullet}  +\sum_{p+q=n+1} \sum_{\sigma(p)=n}  k_q\circ \delta^{\bullet} \circ ( k_p\otimes \Id) \circ \sigma^{\bullet} = 0
		\end{align*}
		 Here, $\sigma$ is a $p$-unshuffle in $S_n$.  Also, because $k_n: L^{\otimes n-1}\otimes M \to M$, we introduce the permutation $\delta$ and use skew-symmetry of $k_n$ in the case $\sigma(p)=n$.  That is, $\delta^{\bullet}$ is the map that permutes the $k_p$ term past the remaining elements:
		 \begin{align*}
		 k_q\big( \underbrace{k_p(x_{\sigma(1)}, \ldots, x_{\sigma(p)})}_{\in M}, x_{\sigma(p+1)}, \ldots, x_{\sigma(n)} \big) &= k_q\Big( \delta^{\bullet} \circ \big( \underbrace{k_p(x_{\sigma(1)}, \ldots, x_{\sigma(p)})}_{\in M}, x_{\sigma(p+1)}, \ldots, x_{\sigma(n)} \big) \Big) \\
		 &=k_q\big( x_{\sigma(p+1)}, \ldots, x_{\sigma(n)}, \underbrace{k_p(x_{\sigma(1)}, \ldots, x_{\sigma(p)})}_{\in M} \big)
		 \end{align*}
\end{definition}

	\begin{example}
		The $n=1$ module relation says that $M$ is a chain complex with differential $k_1$:
		$$k_1(k_1(m)) = 0$$
		The $n=2$ module relation says that the action satisfies the Leibniz rule:
		$$k_2(l_1(x_1),m) + k_2(x_1,k_1(m)) + k_1(k_2(x_1,m)) = 0$$
		Using a different notation, we could also write
		$$[\partial x_1, m] + [x_1,\partial m] + \partial[x_1,m] = 0 $$
		to remind us of differential graded Lie algebras.
	\end{example}




\begin{definition}
	Following \cite{A1}, let $(L, l_i)$ be an $L_\infty$-algebra, and let $(M,k_i)$ and $(M',k_i')$ be two $L_\infty$-modules over $L$.  An \textbf{$L_\infty$-module homomorphism} from $M$ to $M'$ is a collection $\{h_n: L^{\otimes{(n-1)}} \otimes M \to M'\}$ of skew-symmetric multilinear maps of degree $n-1$ such that 
	\begin{align*}
		&\sum_{\substack{i+j=n+1 \\ i<n}} \sum_{\sigma(n)=n} h_j \circ (l_i \otimes \Id) \circ \sigma^{\bullet} \\
		+&\sum_{i+j=n+1} \sum_{\sigma(i)=n} h_j \circ \delta^{\bullet} \circ (k_i \otimes \Id) \circ \sigma^{\bullet} \\ 
	 =& \sum_{r+s=n+1} \sum_{\tau} k_r'\circ (\Id\otimes h_s) \circ (\tau^{\bullet} \otimes \Id)
	\end{align*}
	Here, $\sigma$ is an $i$-unshuffle in $S_n$, and $\tau$ is an $(n-s)$-unshuffle in $S_{n-1}$.

\end{definition}

\begin{remark} Figure \ref{fig:C} is a depiction of the $L_\infty$-module homomorphism relation.  

	\begin{figure}[H]
		\centering
		\scalebox{.9}{\input{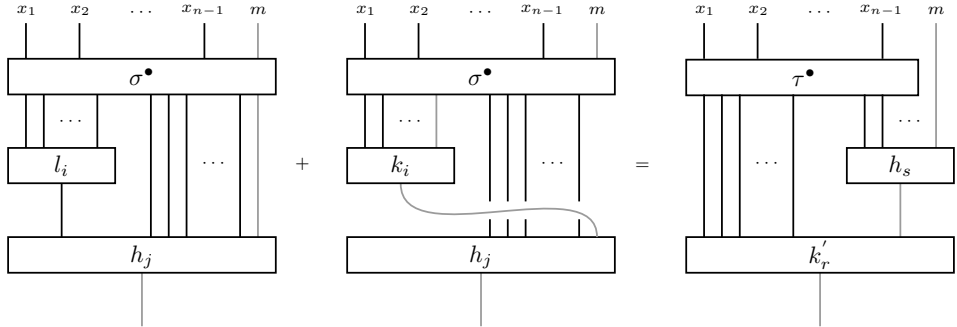}}
    	\caption{A graphical depiction of the $L_\infty$-module homomorphism relation.  This should be interpreted as $\sum h_j \circ (l_i \otimes \Id) \circ \sigma^{\bullet}+\sum h_j \circ \delta^{\bullet} \circ (k_i \otimes \Id) \circ \sigma^{\bullet}
	 = \sum k_r'\circ (\Id\otimes h_s) \circ (\tau^{\bullet} \otimes \Id)$.}
	 \label{fig:C}
	\end{figure}
\end{remark}

\begin{example}
		The $n=1$ module homomorphism relation says that $h_1$ is a chain map:
		$h_1k_1(m) = k_1'h_1(m)$. The $n=2$ module homomorphism relation says:
		$$h_2(l_1(x_1),m) + h_2(x_1,k_1(m)) + h_1(k_2(x_1,m)) = k_2'(x_1,h_1(m)) + k_1'(h_2(x_1,m))$$
	\end{example}



\begin{definition}
	The \textbf{identity map}, $\Id_M$, of an $L_\infty$ module $M$ is defined as follows. $(\Id_M)_1$ is the identity map of the underlying graded vector space $M$, and $(\Id_M)_r = 0$ for $r\geq 2$.  It is straightforward to check that this satisfies the definition of an $L_\infty$-module homomorphism.
\end{definition}

\begin{definition}
	Let $L$ be an $L_\infty$-algebra, and let $A,B,$ and $C$ be $L_\infty$-modules over $L$.  Given $L_\infty$-module homomorphisms $A \xrightarrow{f} B \xrightarrow{g} C$, we define the composition $g\circ f$ by
	$$(g\circ f)_n = \sum_{i+j=n+1} \sum_{\sigma(i)=n} g_j \circ \delta^{\bullet} \circ (f_i \otimes \Id) \circ \sigma^{\bullet}$$
	where $\sigma$ is an $i$-unshuffle in $S_n$, and $\lambda^{\bullet}$ is the map that permutes the module element to the final input.
\end{definition}

\begin{figure}[H]
		\centering
		\scalebox{1}{\input{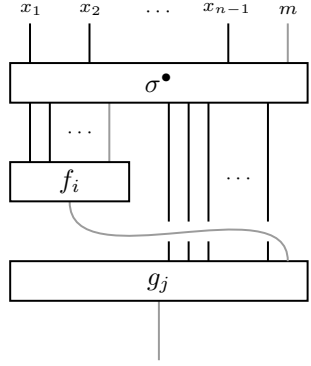}}
    	\caption{A graphical depiction of the composition of two $L_\infty$-module homomorphisms.  This should be interpreted as $(g\circ f)_n = \sum g_j \circ \delta^{\bullet} \circ (f_i \otimes \Id) \circ \sigma^{\bullet}$.} 
	\end{figure}

The following Lemma is perhaps well-known, but we do not know a reference for it.  Pictures representing each step in the proof are given in the appendix.

\begin{lemma}[Composition]
	Let $(L,l_i)$ be an $L_\infty$-algebra, and let $A,B,$ and $C$ be $L_\infty$-modules over $L$, with module operations denoted by $a_i, b_i$, and $c_i$, respectively.  Given $L_\infty$-module homomorphisms $A \xrightarrow{f} B \xrightarrow{g} C$, the composition $g\circ f$ is an $L_\infty$-module homomorphism.
\end{lemma}
\begin{proof} This follows from the fact that both $f$ and $g$ are $L_\infty$-module homomorphisms.  Below, we will apply the $L_\infty$-module homomorphism relation for $f$, then we will apply the $L_\infty$-module homomorphism relation for $g$, and then we will conclude the $L_\infty$-module homomorphism relation for $g\circ f$.

\vspace{1em}
\noindent \textbf{Step 1.} The relation that we need to show is 

	$$\sum_{i+j=n+1} \sum_{\sigma} (g\circ f)_j \circ (a_i \otimes \Id) \circ \sigma^\bullet = \sum_{r+s=n+1} \sum_{\tau} c_r \circ (\Id \otimes (g\circ f)_s) \circ \tau^\bullet$$
	where $\sigma$ is an $(i,n-i)$-unshuffle and $\tau$ is an $(n-s,s-1)$-unshuffle.

\vspace{1em}
\noindent \textbf{Step 2.} Break the left-hand side into two parts, and replace $(g\circ f)_j$ with its definition
	\begin{align*}
		&\sum_{i+j=n+1} \sum_{\sigma(i)=n} \sum_{p+q=j+1} \sum_{\theta(p)=j} g_q \circ \delta^\bullet \circ (f_p \otimes \Id) \circ \theta^\bullet \circ \lambda^\bullet \circ (a_i \otimes \Id) \circ \sigma^\bullet \\
		+ &\sum_{\substack{i+j=n+1 \\ i<n}} \sum_{\substack{\sigma(n)=n}} \sum_{p+q=j+1} \sum_{\theta(p)=j} g_q \circ \delta^\bullet \circ (f_p \otimes \Id) \circ \theta^\bullet \circ (l_i \otimes \Id) \circ \sigma^\bullet
	\end{align*}
	where $\delta^\bullet$ is the map that permutes the module element to the last input.

\vspace{1em}
\noindent \textbf{Step 3.} In the first sum, applying $\sigma^\bullet$ and $\theta^\bullet$ results in a block of size $i$ being inputted to $a_i$, a block of size $p-1$ being inputted into $f_p$, together with the output of $a_i$, and then a block of size $j-p$ remaining elements (which will be inputted into $g_q$).  An equivalent way to achieve this is to first apply a $(p+i-1)$-unshuffle $\eta^\bullet$ and then an $i$-unshuffle $\psi^\bullet$.  If $\eta(p+i-1)=n$ and $\psi(i)=p+i-1$, we again obtain a block of size $i$ being inputted into $a_i$, then a block of size $p-1$ being inputted into $f_p$, together with the output of $a_i$, with $j-p$ elements remaining.  

In the second sum, we do the same thing, except the output of $l_i$ can either go into the first input of $f_p$ or the first input of $g_q$, by the definition of unshuffle.  So we decompose the second sum to reflect these two cases.

	\begin{align*}
		&\sum_{i+j=n+1} \sum_{p+q=j+1} \sum_{\substack{\eta\in S(p+i-1,j-p) \\ \eta(p+i-1)=n}}  \sum_{\substack{\psi \in S(i,p-1) \\ \psi(i)=p+i-1}} g_q \circ \delta^\bullet \circ (f_p \otimes \Id) \circ \lambda^\bullet \circ (a_i \otimes \Id) \circ (\psi^\bullet\otimes \Id)\circ \eta^\bullet \\
		+ &\sum_{\substack{i+j=n+1 \\ i<n}} \sum_{\substack{p+q=j+1 \\ p>1}} \sum_{\substack{\eta\in S(p+i-1,j-p) \\ \eta(p+i-1)=n}}  \sum_{\substack{\psi \in S(i,p-1) \\ \psi(i)=i}} g_q \circ \delta^\bullet \circ (f_p \otimes \Id)  \circ (l_i \otimes \Id) \circ (\psi^\bullet\otimes \Id)\circ \eta^\bullet \\
		+ &\sum_{\substack{i+j=n+1 \\ i<n}} \sum_{p+q=j+1} \sum_{\substack{\eta\in S(p+i,j-p-1) \\ \eta(p+i)=n}}  \sum_{\substack{\psi \in S(i,p-1) \\ \psi(i)=i}} g_q \circ \delta^\bullet \circ (l_i \otimes f_p \otimes \Id) \circ (\psi^\bullet\otimes \Id)\circ \eta^\bullet \\
	\end{align*}

\vspace{1em}
\noindent \textbf{Step 4.} Reindex over $\alpha = p+i$.
	\begin{align*}
		&\sum_{\alpha=2}^{n+1} \sum_{p+i=\alpha} \sum_{\substack{\eta\in S(\alpha-1,n-\alpha+1) \\ \eta(\alpha-1)=n}}  \sum_{\substack{\psi \in S(i,\alpha-1-i) \\ \psi(i)=\alpha-1}} g_{n+2-\alpha} \circ \delta^\bullet \circ (f_p \otimes \Id) \circ \lambda^\bullet \circ (a_i \otimes \Id) \circ (\psi^\bullet\otimes \Id)\circ \eta^\bullet \\
		+ &\sum_{\alpha=2}^{n+1} \sum_{\substack{p+i=\alpha \\ 1<p, i<n}} \sum_{\substack{\eta\in S(\alpha-1,n-\alpha+1) \\ \eta(\alpha-1)=n}}  \sum_{\substack{\psi \in S(i,\alpha-1-i) \\ \psi(i)=i}} g_{n+2-\alpha} \circ \delta^\bullet \circ (f_p \otimes \Id)  \circ (l_i \otimes \Id) \circ (\psi^\bullet\otimes \Id)\circ \eta^\bullet \\
		+ &\sum_{\alpha=2}^{n+1} \sum_{\substack{p+i=\alpha \\ i<n}} \sum_{\substack{\eta\in S(\alpha,n-\alpha) \\ \eta(\alpha)=n}}  \sum_{\substack{\psi \in S(i,p-1) \\ \psi(i)=i}} g_{n+2-\alpha} \circ \delta^\bullet \circ (l_i \otimes f_p \otimes \Id) \circ (\psi^\bullet\otimes \Id)\circ \eta^\bullet \\
	\end{align*}

\vspace{1em}
\noindent \textbf{Step 5.} Apply the module homomorphism relation for $f$ in the first two sums.  In the third sum, change notation from $i$ to $t$ and from $p$ to $s$.

	\begin{align*}
		&\sum_{\alpha=2}^{n+1} \sum_{t+s=\alpha} \sum_{\substack{\eta\in S(\alpha-1,n-\alpha+1) \\ \eta(\alpha-1)=n}}  \sum_{\substack{\tau \in S(t-1,s-1)}} g_{n+2-\alpha} \circ \delta^\bullet \circ (b_t \otimes \Id) \circ (\Id \otimes f_s \otimes \Id) \circ (\tau^\bullet\otimes \Id)\circ \eta^\bullet \\
		+ &\sum_{\alpha=2}^{n+1} \sum_{\substack{t+s=\alpha \\ t<n}} \sum_{\substack{\eta\in S(\alpha,n-\alpha) \\ \eta(\alpha)=n}}  \sum_{\substack{\psi \in S(t,s-1) \\ \psi(t)=t}} g_{n+2-\alpha} \circ \delta^\bullet \circ (l_t \otimes f_s \otimes \Id) \circ (\psi^\bullet\otimes \Id)\circ \eta^\bullet \\
	\end{align*}

\vspace{1em}
\noindent \textbf{Step 6.} In the first sum, combine $\tau \in S(t-1,s-1)$ and $\eta \in S(\alpha-1, n-\alpha+1)$ into a single $(t-1,s,n-\alpha+1)$-unshuffle, denoted by $\pi$.  In the second sum, combine $\psi \in S(t,s-1)$ and $\eta\in S(\alpha,n-\alpha)$ into a single $(t,s,n-\alpha)$-unshuffle, denoted by $\pi$.

	\begin{align*}
		&\sum_{\alpha=2}^{n+1} \sum_{t+s=\alpha} \sum_{\substack{\pi\in S(t-1, s,n-\alpha+1) \\ \pi(\alpha-1)=n}} g_{n+2-\alpha} \circ \delta^\bullet \circ (b_t \otimes \Id) \circ (\Id \otimes f_s \otimes \Id) \circ \pi^\bullet \\
		+ &\sum_{\alpha=2}^{n+1} \sum_{\substack{t+s=\alpha \\ t<n}}  \sum_{\substack{\pi \in S(t,s,n-\alpha) \\ \pi(\alpha)=n}} g_{n+2-\alpha} \circ \delta^\bullet \circ (l_i \otimes f_s \otimes \Id) \circ \pi^\bullet \\
	\end{align*}

\vspace{1em}
\noindent \textbf{Step 7.} In the first sum, $\pi$ unshuffles the $n$ elements into a block of size $t-1$, a block of size $s$, and a block of size $n-\alpha+1$.  The block of size $s$ is then inputted into $f_s$, and then the output of $f_s$ is then inputted into $b_t$, as the module element, with the block of size $t-1$.    

An equivalent way of achieving this is to apply an $(n-s, s-1)$-unshuffle to the $(n-1)$-algebra elements, to form blocks of size $(n-s)$ and $s-1$, and then input the $s-1$ algebra elements into $f_s$, with the module element.  Then, apply an $t$-unshuffle $\sigma^\bullet$ to these $n-s+1$ elements.  By requiring $\sigma(t)=n-s+1$, we obtain a block of size $t-1$, plus a module element, that we input into $b_t$.   We can do an analagous reformulation of the second sum.

\begin{align*}
		&\sum_{\alpha=2}^{n+1} \sum_{t+s=\alpha} \sum_{\substack{\phi \in S(n-s, s-1)}} \sum_{\substack{\sigma \in S(t, n-s+1) \\ \sigma(t)=n-s+1}} g_{n+2-\alpha} \circ \delta^\bullet \circ (b_t \otimes \Id) \circ \sigma^\bullet \circ (\Id \otimes f_s) \circ (\phi^\bullet \otimes \Id) \\
		+ &\sum_{\alpha=2}^{n+1} \sum_{\substack{t+s=\alpha \\ t<n}} \sum_{\substack{\phi \in S(n-s, s-1)}} \sum_{\substack{\sigma \in S(t, n-s+1) \\ \sigma(n-s+1)=n-s+1}} g_{n+2-\alpha} \circ (l_t \otimes \Id) \circ \sigma^\bullet \circ (\Id \otimes f_s) \circ (\phi^\bullet \otimes \Id) \\
	\end{align*}

\vspace{1em}
\noindent \textbf{Step 8.} Reindex, noting that $\sum_{\alpha=2}^{n+1} \sum_{t+s=\alpha} = \sum_{s=1}^{n} \sum_{t=1}^{n+1-s} = \sum_{s=1}^{n} \sum_{x+y=n+2-s}$.

\begin{align*}
		&\sum_{s=1}^{n} \sum_{x+y=n+2-s} \sum_{\substack{\phi \in S(n-s, s-1)}} \sum_{\substack{\sigma \in S(x, n-s+1) \\ \sigma(x)=n-s+1}} g_{y} \circ \delta^\bullet \circ (b_x \otimes \Id) \circ \sigma^\bullet \circ (\Id \otimes f_s) \circ (\phi^\bullet \otimes \Id) \\
		+ &\sum_{s=1}^{n} \sum_{\substack{x+y=n+2-s \\ x < n}} \sum_{\substack{\phi \in S(n-s, s-1)}} \sum_{\substack{\sigma \in S(x, n-s+1) \\ \sigma(n-s+1)=n-s+1}} g_{y} \circ (l_x \otimes \Id) \circ \sigma^\bullet \circ (\Id \otimes f_s) \circ (\phi^\bullet \otimes \Id) \\
	\end{align*}

\vspace{1em}
\noindent \textbf{Step 9.} Apply the morphism relation for $g$.
\begin{align*}
	\sum_{s=1}^{n} \sum_{r+q=n-s+2} \sum_{\substack{\phi \in S(n-s,s-1)}} \sum_{\kappa \in S(r-1,q-1)} c_r \circ (\Id \otimes g_q) \circ  (\kappa^\bullet \otimes \Id) \circ (\Id \otimes f_s)  \circ (\phi^\bullet \otimes \Id)
\end{align*}

\vspace{1em}
\noindent \textbf{Step 10.} Combine $\kappa$ and $\phi$ into a single permutation $\pi$.

\begin{align*}
	\sum_{s=1}^{n} \sum_{r+q=n-s+2} \sum_{\substack{\pi \in S(r-1,q-1,s-1)}} c_r \circ (\Id \otimes g_q) \circ (\Id \otimes f_s)  \circ (\pi^\bullet \otimes \Id)
\end{align*}

\vspace{1em}
\noindent \textbf{Step 11.} Split $\pi$ into $\tau$ and $\psi$.  The map $\lambda^\bullet$ is needed to permute the module element into the last input of $g_q$.

\begin{align*}
	\sum_{s=1}^{n} \sum_{r+q=n-s+2} \sum_{\substack{\tau \in S(r-1,n-r)}} \sum_{\substack{\psi \in S(s,q-1)}} c_r \circ \big(\Id \otimes \big[g_q \circ \lambda^\bullet \circ (f_s \otimes \Id) \circ \psi^\bullet \big] \big) \circ (\tau^\bullet \otimes \Id)
\end{align*}

\vspace{1em}
\noindent \textbf{Step 12.} Change how we index over $s,r,q$.

\begin{align*}
	\sum_{r=1}^{n} \sum_{s+q=n+2-r} \sum_{\substack{\tau \in S(r-1,n-r)}} \sum_{\substack{\psi \in S(s,q-1)}} c_r \circ \big(\Id \otimes \big[g_q \circ \lambda^\bullet \circ (f_s \otimes \Id) \circ \psi^\bullet \big] \big) \circ (\tau^\bullet \otimes \Id)
\end{align*}

\vspace{1em}
\noindent \textbf{Step 13.} Use the definition of $g\circ f$.

\begin{align*}
	\sum_{r=1}^{n} \sum_{\substack{\tau \in S(r-1,n-r)}} c_r \circ (\Id \otimes (g\circ f)_{n+1-r}) \circ (\tau^\bullet \otimes \Id)
\end{align*}

\vspace{1em}
\noindent \textbf{Step 14.} This is

$$\sum_{r+s=n+1} \sum_{\tau \in S(r-1,s-1)} c_r \circ (\Id \otimes (g\circ f)_s) \circ (\tau^\bullet \otimes \Id)$$

\end{proof}

\section{Restriction of Scalars}

In this section, we prove the main result.  We start by defining the restriction of scalars functor on objects and prove that the result is an $L_\infty$-module.  We then define the restriction of scalars functor on morphisms, proving that the result is an $L_\infty$-module homomorphism.  Finally, we complete the proof of functoriality.  The end of this section contains a technical lemma that is applied several times throughout the aforementioned proofs.

\begin{lemma}[Objects]
	\label{Lem:Objects}
	Suppose $I: (L',l') \to (L,l)$ is a map of $L_\infty$-algebras.  If $(M,k)$ is an $L$-module, then $I^*M:=(M,k')$ is an $L'$-module, where $k'_n: L^{\otimes n-1} \otimes M \to M$ is given by
	\begin{align*}
				k_n' &=\sum_{r=1}^{n-1} \sum_{\substack{\tau\in S'(i_1,\ldots,i_r) \\ i_1+\ldots+i_r=n-1}} k_{r+1} \circ (I_{i_1} \otimes \cdots \otimes I_{i_r} \otimes \Id) \circ (\tau^\bullet \otimes \Id)
	\end{align*}
\end{lemma}
\begin{proof}

The idea of the proof is straightforward.  We will first make a substitution using the definition of $k'$ (steps 1-2).  We will then use the $L_\infty$-algebra homomorphism relation for $I$ to exchange any $I$ and $l'$ terms (steps 3-9). The terms that remain will then cancel by applying the $L_\infty$-module relation for $k$ (steps 10-19).  Pictures representing each step in the proof are given in the appendix. 

	\vspace{1em}
	\noindent \textbf{Step 1.} The $L_\infty$ relation for $k'_n$ that we need to show is zero is: 
		\begin{align*}
			\sum_{\substack{p+q=n+1 \\ p<n}}\sum_{\sigma(n)=n} k'_q\circ (l'_p \otimes \Id) \circ \sigma^\bullet 
			+
			\sum_{p+q=n+1} \sum_{\sigma(p)=n}  k'_q\circ \delta^\bullet \circ ( k'_p\otimes \Id) \circ \sigma^\bullet
			= 0
		\end{align*}

\vspace{1em}
	\noindent \textbf{Step 2.} Focusing only on the first double sum for now, we substitute for $k_q'$ using its definition:

	$$
	\sum_{\substack{p+q=n+1 \\ p<n}}\sum_{\sigma(n)=n} \sum_{\substack{\tau\in S'(i_1,\ldots,i_r) \\ 1\leq r\leq q-1\\ i_1+\ldots+i_r=q-1}} k_{r+1}\bigg((I_{i_1}\otimes \cdots \otimes I_{i_r} \otimes \Id) \circ (\tau^\bullet \otimes \Id) \circ (l'_p \otimes \Id) \circ \sigma^\bullet\bigg)
	$$

	\noindent \textbf{Step 3.} The goal now is to use the morphism relation to commute the $l'_p$ and $I$ terms.  To do so, we will break down this sum by the specific morphism relation that we will apply $(k=1,\ldots, n-1)$.  In particular, this is determined by the sum of $p$ and the size of the block to which $\tau$ sends $l'_p$.  We will denote the block containing $l'_p$ by $i_l$, and we will denote its size by $s$.
	\begin{align*}
	&\sum_{p=1}^{n-1} \sum_{\sigma(n)=n} \sum_{\substack{\tau\in S'(i_1,\ldots,i_r) \\ 1\leq r\leq n-p\\ i_1+\ldots+i_r=n-p}} k_{r+1}\bigg((I_{i_1}\otimes \cdots \otimes I_{i_r} \otimes \Id) \circ (\tau^\bullet \otimes \Id) \circ (l'_p \otimes \Id) \circ \sigma^\bullet\bigg) \\
	=&\sum_{p=1}^{n-1} \sum_{\sigma(n)=n} \sum_{s=1}^{n-p} \sum_{\substack{\tau\in S'(i_1,\ldots,i_r) \\ 1\leq r\leq n-p\\ i_1+\ldots+i_r=n-p \\ i_l=s}}  k_{r+1}\bigg((I_{i_1}\otimes \cdots I_{i_l} \otimes \cdots \otimes I_{i_r} \otimes \Id) \circ (\tau^\bullet \otimes \Id) \circ (l'_p \otimes \Id) \circ \sigma^\bullet\bigg)
	\end{align*}

	We can now reindex over the sum of $p$ and $s$ (on the $(p,s)$-plane, this is summing over the diagonal) to obtain
	$$
	\sum_{k=1}^{n-1} \sum_{p+s=k+1}\sum_{\sigma(n)=n} \sum_{\substack{\tau\in S'(i_1,\ldots,i_r) \\ 1\leq r\leq n-p\\ i_1+\ldots+i_r=n-p \\ i_l= s}}  k_{r+1}\bigg((I_{i_1}\otimes \cdots \otimes I_{i_l} \otimes \cdots \otimes I_{i_r} \otimes \Id) \circ (\tau^\bullet \otimes \Id) \circ (l'_p \otimes \Id) \circ \sigma^\bullet\bigg)
	$$

	\vspace{1em}

\noindent \textbf{Step 4.} Here, we change $\tau$ to $\tau'$ and introduce $\lambda$.  Since $\tau$ is an unshuffle, we can make two observations.  First, $\tau$ sends $l'_p$ to the first input of $I_{i_l}$.  Second, in the partition $i_1+\ldots+i_r=n-p$, the block $i_l$ is the first of its size (i.e. $t<l$ implies $i_t < i_l$), since the first elements of blocks of the same size are in order.  This information allows us to remove $l'_p$ as an input to $\tau$, and then put it back in the correct spot after the remaining elements are permuted.  That is, $\tau$ corresponds to an $(i_1, \ldots, i_l-1, \ldots, i_r)$-unshuffle $\tau'$ in $S_{n-p-1}$, and we will send $l'_p$ to the first input of $I_{i_l}$ via a permutation $\lambda$ after we apply $\tau'$.  Special care is needed when $s=1$, in which case $\tau' \in S(0,i_2, \ldots, i_r)$, and no element will go to the block of size 0.  

Note: because $\tau\in S'(i_1, \ldots, i_r)$, we had conditions that $i_1 \leq \cdots \leq i_r$ and that the order of the first elements among these blocks is preserved.  In the rest of the proof, we must remember these restrictions inherited from $\tau$. We obtain,

	\begin{equation*}
	\begin{split}
	\sum_{k=1}^{n-1}  \sum_{p+s=k+1} \ & \sum_{\substack{\sigma(n)=n}} \   \sum_{\substack{\tau'\in S(i_1,i_2,\ldots,i_l-1, \ldots, i_r)\\ i_1+\ldots+i_r=n-p \\ i_l = s}} \\ 
	&k_{r+1}\bigg((I_{i_1} \otimes \cdots \otimes I_{i_l} \otimes \cdots \otimes I_{i_r} \otimes \Id) \circ \lambda^\bullet \circ (\Id \otimes \tau'^\bullet \otimes \Id) \circ (l'_p \otimes \Id) \circ \sigma^\bullet \bigg)
	\end{split}
	\end{equation*}

\vspace{2em}

\noindent \textbf{Step 5.} Combine $\sigma$ and $\tau'$ into $\psi$.  Now we observe that applying a $p$-unshuffle and then $\tau'$ to the remaining inputs is equivalent to doing a $(p,i_1, \ldots, i_r)$-unshuffle to all of the inputs at once.  We obtain

	$$ \displaystyle{\sum_{k=1}^{n-1} \sum_{p+s=k+1} \ \sum_{\substack{\psi\in S(p,i_1,\ldots,i_l-1,\ldots,i_r,1) \\ i_1+\ldots+i_r=n-p \\ i_l=s \\ \psi(n)=n}}  k_{r+1}\bigg( (I_{i_1} \otimes \cdots \otimes I_{i_l} \otimes \cdots \otimes I_{i_r}) \circ \lambda^{\bullet}  \circ (l'_p \otimes \Id) \circ \psi^\bullet \bigg)}$$

\vspace{2em}
\noindent \textbf{Step 6.} Change from $\psi$ to $\mu, \alpha, \omega$.  Notice that a $(p,i_1, \ldots, i_l-1,\ldots i_r)$-unshuffle is the same as first doing a $(p+i_l-1)$-unshuffle, and then doing a $(p,i_l-1)$-unshuffle on the $(p+i_l-1)$-block and an $(i_1, \ldots, \widehat{i_l},\ldots, i_r)$-unshuffle on the rest.  Since we are fixing $i_l=s$, note that $p+i_l-1=k$.

	Afterwards, we need to apply a permutation $\omega$ to move the strands in the $i_l$ block back to their original position between the $i_{l-1}$ and $i_{l+1}$ blocks.  That is, $\omega$ is the block permuation so that applying $\omega^\bullet$ to the blocks $\{1,i_l-1, i_1, \ldots, \widehat{i_l}, \ldots, i_r\}$ yields $\{1,i_1, \ldots, i_l-1,\ldots i_r\}$.  We apply $\lambda^\bullet$ after $\omega^\bullet$ to move the $l'_p$ term.

	\begin{equation*}
	\begin{split}
	\sum_{k=1}^{n-1} & \sum_{p+s=k+1}  \  \sum_{\substack{\mu \in S(k, i_1, \ldots, \widehat{i_l}, \ldots, i_r, 1) \\ i_1 + \cdots + i_r = n-p \\i_l=s \\ \mu(n)=n}} \  \sum_{\alpha \in S(p,k-p)}   \\ 
	&k_{r+1}\bigg((I_{i_1} \otimes \cdots \otimes I_{i_l} \otimes \cdots \otimes I_{i_r}  \otimes \Id) \circ \lambda^\bullet \circ \omega^\bullet \circ (l'_p \otimes \Id) \circ (\alpha^\bullet \otimes \Id) \circ \mu^\bullet \bigg)
	\end{split}
	\end{equation*}

\vspace{2em}
\noindent \textbf{Step 7.} Since $k_{r+1}$ is skew-symmetric, we can move the $I_{i_l}$ term to the first input.

\begin{equation*}
	\begin{split}
	 \sum_{k=1}^{n-1} & \sum_{p+s=k+1} \ \sum_{\substack{\mu \in S(k, i_1, \ldots, \widehat{i_l}, \ldots, i_r, 1) \\ i_1 + \cdots + i_r = n-p \\i_l=s \\ \mu(n)=n}}  \ \sum_{\alpha \in S(p,k-p)} \\
	 &  k_{r+1}\bigg((I_{i_l} \otimes I_{i_1} \otimes \cdots \otimes \widehat{I_{i_l}}  \otimes \cdots \otimes I_{i_r}  \otimes \Id) \circ (l'_p \otimes \Id) \circ (\alpha^\bullet \otimes \Id) \circ \mu^\bullet \bigg)
	\end{split}
	\end{equation*}

\vspace{2em}
\noindent \textbf{Step 8.} Rewrite the maps as

\begin{equation*}
	\begin{split}
	\sum_{k=1}^{n-1} & \sum_{p+s=k+1} \  \sum_{\substack{\mu \in S(k, i_1, \ldots, \widehat{i_l}, \ldots, i_r, 1) \\ i_1 + \cdots + i_r = n-p \\i_l=s \\ \mu(n)=n}} \  \sum_{\alpha \in S(p,k-p)}  \\ 
	& k_{r+1}\bigg( \big[I_{i_l} \circ (l'_p \otimes \Id) \circ \alpha^\bullet \big] \otimes \big[(I_{i_1} \otimes \cdots \otimes \widehat{I_{i_l}}  \otimes \cdots \otimes I_{i_r}  \otimes \Id) \big] \circ \mu^\bullet \bigg)
	\end{split}
	\end{equation*}

\vspace{2em}
\noindent \textbf{Step 9.} Apply the $L_\infty$-algebra homomorphism relation to the terms $I_{i_l} \circ (l'_p \otimes \Id) \circ \alpha^\bullet$.  Since we no longer are keeping track of $p$, we also use the fact that $p+s=k+1$ to rewrite the conditions for $\mu$.

\begin{align*}
	\sum_{k=1}^{n-1} & \sum_{\substack{1 \leq t \leq k \\ a_1 + \ldots +a_t = k \\ a_r \geq 1}}  \sum_{\gamma \in S'(a_1, \ldots, a_t)} \sum_{\substack{\mu \in S(k, i_1, \ldots, \widehat{i_l}, \ldots, i_r, 1) \\ i_1 + \cdots \widehat{i_l} + \cdots + i_r = n-1-k \\ \mu(n)=n}} \\
	&k_{r+1} \bigg( \big[ l_t \circ (I_{a_1} \otimes \cdots \otimes I_{a_t}) \circ \gamma^\bullet \big] \otimes  \big[(I_{i_1} \otimes \cdots \otimes \widehat{I_{i_l}}  \otimes \cdots \otimes I_{i_r} \otimes \Id) \big] \circ \mu^\bullet \bigg)
\end{align*}

\vspace{2em}
\noindent \textbf{Step 10.} Rewrite the maps as

\begin{equation*}
	\begin{split}
	\sum_{k=1}^{n-1} & \sum_{\substack{1 \leq t \leq k \\ a_1 + \ldots +a_t = k \\ a_r \geq 1}}  \sum_{\gamma \in S'(a_1, \ldots, a_t)} \sum_{\substack{\mu \in S(k, i_1, \ldots, \widehat{i_l}, \ldots, i_r, 1) \\ i_1 + \cdots \widehat{i_l} + \cdots + i_r = n-1-k \\ \mu(n)=n}} \\
	& k_{r+1} \Bigg( (l_t\otimes \Id) \circ (I_{a_1}\otimes \cdots \otimes I_{a_t} \otimes I_{i_1} \otimes \cdots \otimes \widehat{I_{i_l}} \otimes \cdots \otimes I_{i_r}) \circ (\gamma^\bullet \otimes \Id) \circ \mu^\bullet \Bigg) 
\end{split}
\end{equation*}

\vspace{2em}
\noindent \textbf{Step 11.} We can combine $\mu$ and $\gamma$ into one permutation $\eta$.  Indeed, applying $\mu$ and then an $(a_1, \ldots, a_t)$-unshuffle on the $k$-block is the same as applying an $(a_1, \ldots, a_t, i_1, \ldots, \widehat{i_l}, \ldots, i_r,1)$-unshuffle all at once.

	$$\sum_{k=1}^{n-1}\sum_{\substack{\eta \in S(a_1, \ldots, a_t, i_1, \ldots,\widehat{i_l}, \ldots i_r,1)\\ 1\leq t \leq k \\ a_1 + \ldots + a_t = k \\ i_1 + \cdots + \widehat{i_l} + \cdots + i_r = n-1-k \\ \eta(n)=n}}  k_{r+1} \bigg((l_t\otimes \Id) \circ (I_{a_1} \otimes \cdots \otimes I_{a_t} \otimes I_{i_1} \otimes \cdots \otimes \widehat{I_{i_l}} \otimes \cdots \otimes I_{i_r} \otimes \Id) \circ \eta^\bullet\bigg)$$

\vspace{2em}
\noindent \textbf{Step 12.}  Since $k=1, \ldots, n-1$, we can drop the sum over $k$ from the notation and just require that $a_1, \ldots, a_t, i_1, \ldots, i_r$ is a partition of $n-1$, with $a_1\leq \ldots \leq a_t$, $i_1 \leq \ldots \leq i_r$, and $t\geq 1$ and $r\geq 1$.  If we fix $\eta \in S(a_1, \ldots, a_t, i_1, \ldots, \widehat{i_{l}}, \ldots, i_r)$, we don't have any relation between the two partitions $a_1\leq \ldots \leq a_t$ and $i_1 \leq \ldots \leq i_r$.  That is, the sizes of the blocks are in order as part of their respective partitions, but it might not be the case that $a_1, \ldots, a_t, i_1, \ldots, i_r$ is in increasing order as a whole.  However, from these two partitions, we can use an unshuffle to construct a new partition where the sizes of the boxes are in order.  Indeed, define $\sigma$ so that $(\sigma^{-1})^\bullet$ arranges the $a_1, \ldots, a_t, i_1, \ldots, i_r$ in increasing order (to get a unique $\sigma$, require that the order of the $a$'s is preserved, the order of the $i$'s is preserved, and that, using $\eta$, the first elements of boxes of same size are in order).  Then let $c_1, \ldots, c_\alpha:= (\sigma^\bullet)^{-1}(a_1, \ldots, a_t, i_1, \ldots, i_r)$.  To summarize, what we have done is define a new partition $c_1, \ldots, c_\alpha$ of $n-1$ so that $c_{\sigma(1)} = a_1, \ldots, c_{\sigma(t)} = a_t, c_{\sigma(t+1)} = i_1, \ldots, c_{\alpha} = i_r$.  Of course, since $a_1\leq \ldots \leq a_t$, $\sigma$ is a $t$-unshuffle. Moreover, we define $\tau$ by requiring that the elements that $\eta$ puts into the $a_1, \ldots, a_t$ and $i_1, \ldots, i_r$-boxes are precisely those that $\tau$ puts into the $c_{\sigma(1)}, \ldots, c_{\sigma(t)}$ and $c_{\sigma(t+1)}, \ldots, c_{\alpha}$-boxes, respectively.  Finally, since $\alpha = t+r-1$, we relabeled $k_{r+1}$ as $k_{\alpha+2-t}$.  Note that we can reverse this whole construction to obtain an inverse correspondence.  This process is similar to Lemma 4.

	$$\sum_{\substack{\tau \in S'(c_1, \ldots, c_{\alpha}) \\ c_1 + \ldots + c_{\alpha} = n-1 \\ 1 \leq \alpha \leq n-1}}  \   \sum_{\substack{\sigma \in S(t,\alpha+1-t) \\ \sigma(\alpha+1)=\alpha+1 \\ 1 \leq t \leq \alpha + 1}} k_{\alpha+2-t} \circ (l_t \otimes \Id) \circ \sigma^\bullet  \circ (I_{c_1} \otimes \cdots \otimes I_{c_\alpha}\otimes \Id) \circ (\tau^\bullet \otimes \Id)$$

\vspace{1em}
\noindent \textbf{Step 13.} On the other hand, we now examine the second term in the original sum:
$$
\sum_{p+q=n+1}\sum_{\sigma(p)=n} k'_q\circ \delta^\bullet \circ ( k'_p\otimes \Id) \circ \sigma^\bullet
$$

\vspace{1em}
\noindent \textbf{Step 14.}  Use the definition of $k'$ to substitute for $k_p'$ and $k_q'$.  The cases $p=1$ and $q=1$ require some care; they correspond to the cases $r=0$ and $s=0$, respectively.  If $r=0$, then $\phi=\Id$, and if $s=0$, then $\psi=\Id$.  We also disallow $r$ and $s$ to be zero simultaneously.  

\begin{equation*}
\begin{split}
	&\sum_{\substack{p+q=n+1 \\ 1\leq p \leq n}} \sum_{\sigma(p)=n} \sum_{\substack{\phi\in S'(i_1,\ldots,i_r) \\ 0 \leq r \leq p-1 \\ i_1+\ldots+i_r=p-1}}\sum_{\substack{\psi\in S'(j_1,\ldots,j_s) \\ 0 \leq s\leq n-p\\  j_1+\ldots+j_s=n-p}}   \\ 
	& k_{s+1} \circ (I_{j_1}\otimes \cdots \otimes I_{j_s}\otimes \Id) \circ (\psi \otimes \Id) \circ \delta^\bullet \circ (k_{r+1} \otimes \Id) \circ (I_{i_1}\otimes \cdots \otimes I_{i_r} \otimes \Id \big) \circ (\phi^\bullet \otimes \Id) \circ \sigma^\bullet
\end{split}
\end{equation*}

\vspace{1em}
\noindent \textbf{Step 15.}  Commuting composition and tensor product, and replacing $\delta$ with an analogous $\delta'$ that ensures the module element is in the correct spot, we get

\begin{equation*}
\begin{split}
\sum_{\substack{p+q=n+1\\1\leq p\leq n}}\sum_{\sigma(p)=n} & \sum_{\substack{\phi\in S'(i_1,\ldots,i_r) \\ 0 \leq r \leq p-1 \\ i_1+\ldots+i_r=p-1}} \sum_{\substack{\psi\in S'(j_1,\ldots,j_s) \\ 0 \leq s \leq n-p \\ j_1+\ldots+j_s=n-p}} \\
&  k_{s+1} \circ \delta'^\bullet \circ (k_{r+1} \otimes \Id) \circ (I_{i_1}\otimes \cdots I_{i_r} \otimes \Id \otimes I_{j_1}\otimes \cdots \otimes I_{j_s}) \circ (\phi^\bullet \otimes \Id \otimes \psi^\bullet) \circ \sigma^\bullet  \\
\end{split}
\end{equation*}

\vspace{1em}
\noindent \textbf{Step 16.}  Instead of summing over $r$ and $s$ separately, we can sum over the diagonal $\alpha = r+s$.  

\begin{equation*}
\begin{split}
\sum_{p=1}^n \sum_{\sigma(p)=n} & \sum_{\substack{1 \leq \alpha \leq n-1 \\ r+s=\alpha \\ r,s \geq 0}} \sum_{\substack{\phi\in S'(i_1,\ldots,i_r) \\ i_1+\ldots+i_r=p-1}} \sum_{\substack{\psi\in S'(j_1,\ldots,j_s) \\ j_1+\ldots+j_s=n-p}} \\ 
&k_{s+1} \circ \delta'^\bullet \circ (k_{r+1} \otimes \Id) \circ (I_{i_1}\otimes \cdots I_{i_r} \otimes \Id \otimes I_{j_1}\otimes \cdots \otimes I_{j_s}) \circ (\phi^\bullet \otimes \Id \otimes \psi^\bullet) \circ \sigma^\bullet
\end{split}
\end{equation*}

\vspace{1em}
\noindent \textbf{Step 17.}  Apply Lemma 4, where $r+1$ above corresponds to $t$ below.

$$
\sum_{\substack{\tau \in S'(c_1,\ldots,c_\alpha) \\ c_1+\ldots+c_\alpha=n-1}}  \sum_{\substack{\sigma \in S(t,\alpha+1-t) \\ \sigma(t)=\alpha+1 \\ 1 \leq t \leq \alpha+1}}   k_{\alpha+2-t}\circ \delta'^\bullet \circ (k_{t} \otimes \Id) \circ (\sigma^\bullet\otimes \Id) \circ (I_{c_1} \otimes \cdots \otimes I_{c_\alpha}\otimes \Id) \circ (\tau^\bullet \otimes \Id)
$$

\noindent \textbf{Step 18.} Summarizing what we've done so far, we've shown that the original sum 
\begin{align*}
			&\sum_{p+q=n+1}\sum_{\sigma(n)=n} k'_q\circ (l'_p \otimes \Id) \circ \sigma^\bullet  \\
			&+\\
			&\sum_{p+q=n+1}\sum_{\sigma(p)=n} k'_q\circ (\Id \otimes k'_p) \circ \sigma^\bullet
		\end{align*}
can be rewritten as

\begin{align*}
	&\sum_{\substack{\tau \in S'(c_1, \ldots, c_{\alpha}) \\ c_1 + \ldots + c_{\alpha} = n-1}}  \   \sum_{\substack{\sigma \in S(t,\alpha+1-t) \\ \sigma(\alpha+1)=\alpha+1 \\ 1 \leq t \leq \alpha+1}} k_{\alpha+2-t} \circ (l_t \otimes \Id) \circ \sigma^\bullet  \circ (I_{c_1} \otimes \cdots \otimes I_{c_\alpha}\otimes \Id) \circ (\tau^\bullet \otimes \Id)\\
	+&\sum_{\substack{\tau \in S'(c_1,\ldots,c_\alpha) \\ c_1+\ldots+c_\alpha=n-1}}  \sum_{\substack{\sigma \in S(t,\alpha+1-t) \\ \sigma(t)=\alpha+1 \\ 1 \leq t \leq \alpha+1}}   k_{\alpha+2-t}\circ \delta'^\bullet \circ (k_{t} \otimes \Id) \circ \sigma^\bullet \circ (I_{c_1} \otimes \cdots \otimes I_{c_\alpha}\otimes \Id) \circ (\tau^\bullet \otimes \Id)
\end{align*}

\vspace{1em}
\noindent \textbf{Step 19.} Letting $F=(I_{c_1} \otimes \cdots \otimes I_{c_\alpha}\otimes \Id) \circ (\tau^\bullet\otimes \Id)$ and setting $u=\alpha+2-t$, this becomes
\begin{align*}
	&\sum_{\substack{\tau \in S'(c_1, \ldots, c_{\alpha}) \\ c_1 + \ldots + c_{\alpha} = n-1}} \sum_{t+u=\alpha+2}   \sum_{\substack{\sigma \in S(t,\alpha+1-t) \\ \sigma(\alpha+1)=\alpha+1}} k_{u} \circ (l_t' \otimes \Id) \circ \sigma^\bullet  \circ F \\
	+&\sum_{\substack{\tau \in S'(c_1,\ldots,c_\alpha) \\ c_1+\ldots+c_\alpha=n-1}}  \sum_{t+u=\alpha+2} \sum_{\substack{\sigma \in S(t,\alpha+1-t) \\ \sigma(t)=\alpha+1}}   k_{u}\circ \delta'^\bullet \circ (k_{t} \otimes \Id) \circ \sigma^\bullet \circ F
\end{align*}
which cancel by the module relation. 

\end{proof}

	\begin{lemma}[Morphisms]
	Suppose $L$ and $L'$ are $L_\infty$-algebras and $M$ and $N$ are $L$-modules.  Let $I: L' \to L$ be an $L_\infty$-algebra homomorphism, and let $f: M \to N$ be an $L$-module homomorphism.  Set $(I^*f)_1 = f_1$, and for $n\geq 2$, define $$(I^*f)_n: (L')^{\otimes n-1} \otimes I^*M \to I^*N$$ 
 by 
	$$(I^*f)_n = \sum_{r=1}^{n-1} \sum_{\substack{\tau \in S'(i_1,\ldots, i_r) \\ i_1 +\ldots + i_r = n-1}} f_{r+1} \circ (I_{i_1} \otimes \cdots \otimes I_{i_r} \otimes \Id) \circ (\tau^\bullet \otimes \Id)$$
	Then $I^*f: I^*M \to I^*N$ is a homomorphism of $L'$-modules.
	\end{lemma}

\begin{proof}

We will start by examining the $L_\infty$-module homomorphism relation.  After replacing $I^*f$ and $m'_i$ with their definitions on the left-hand side (steps 1-4), we will rearrange the sum (steps 5-6) and apply the $L_\infty$-algebra relation for $I$ (step 7).  We then rewrite the terms (steps 8-9) and apply the module homomorphism relation for $f$ (step 10).  We then show that the result is equal to the right-hand side (steps 11-16).

\vspace{1em}
\noindent \textbf{Step 1.} To start, we will denote the operations of $M,N,I^*M,I^*N$ by $m,n,m',n'$ respectively.  To show that $I^*f$ is a homomorphsism of $L'$-modules, we must show that it satisfies the $L_\infty$-module homomorphism relation
	$$\sum_{i+j=n+1} \sum_{\sigma} (I^*f)_j \circ (m'_i \otimes \Id) \circ \sigma^\bullet = \sum_{r+s=n+1} \sum_{\tau} n'_r \circ (\Id \otimes (I^*f)_s) \circ \tau^\bullet$$
	where $\sigma$ is an $(i,n-i)$-unshuffle and $\tau$ is an $(n-s,s-1)$-unshuffle.

\vspace{1em}
\noindent \textbf{Step 2.}
	Focusing only on the left-hand side, we break this sum up into two parts
	\begin{align*}
		&\sum_{i+j=n+1} \sum_{\sigma(i)=n} (I^*f)_j \circ \lambda^\bullet \circ (m'_i \otimes \Id) \circ \sigma^\bullet \\
		+ &\sum_{\substack{i+j=n+1 \\ i<n}} \sum_{\substack{\sigma(n)=n}} (I^*f)_j \circ (l'_i \otimes \Id) \circ \sigma^\bullet
	\end{align*}
	where we use skew-symmetry and introduce the permutation $\lambda$ to insert the module element in the correct spot.

\vspace{1em}
\noindent \textbf{Step 3.}
	Replace $I^*f$ with its definition.  Note that we've allowed $r=0$ in the first sum to include the case $j=1$, which corresponds to $f_1 \circ m'_n \circ \sigma^\bullet$.  If $j$ is anything but 1, $r=0$ makes no contribution to the sum.

\vspace{1em}

	\begin{align*}
		&\sum_{\substack{i+j=n+1}} \sum_{\sigma(i)=n} \sum_{r=0}^{j-1} \sum_{\substack{\tau \in S'(i_1,\ldots, i_r) \\ i_1 +\ldots + i_r = j-1}} f_{r+1} \circ (I_{i_1} \otimes \cdots \otimes I_{i_r} \otimes \Id) \circ (\tau^\bullet \otimes \Id) \circ \lambda^\bullet \circ (m'_i \otimes \Id) \circ \sigma^\bullet \\
		+&\sum_{\substack{i+j=n+1 \\ 1 \leq i < n}} \sum_{\sigma(n)=n} \sum_{r=1}^{j-1} \sum_{\substack{\tau \in S'(i_1,\ldots, i_r) \\ i_1 +\ldots + i_r = j-1}} f_{r+1} \circ (I_{i_1} \otimes \cdots \otimes I_{i_r}\otimes \Id) \circ (\tau^\bullet \otimes \Id) \circ (l'_i \otimes \Id) \circ \sigma^\bullet\\
	\end{align*}

\vspace{1em}
\noindent \textbf{Step 4.}
	Now focus on the first sum and replace $m_i'$ with its definition.  Similar to the above, we've allowed for the case $s=0$ to include the case $i=1$, which corresponds to $(I^*f)_n \circ  \lambda^\bullet \circ (m'_1 \otimes \Id) \circ \sigma^\bullet$.  If $i$ is anything but 1, $s=0$ makes no contribution to the sum.

\begin{equation*}
	\begin{split}
		&\sum_{\substack{i+j=n+1}} \sum_{\sigma(i)=n} \sum_{r=0}^{j-1} \sum_{\substack{\tau \in S'(i_1,\ldots, i_r) \\ i_1 +\ldots + i_r = j-1}} \sum_{s=0}^{i-1} \sum_{\substack{\psi\in S'(j_1,\ldots,j_a) \\ j_1+\ldots+j_s=i-1}} \\ 
		&\quad f_{r+1} \circ (I_{i_1} \otimes \cdots \otimes I_{i_r} \otimes \Id) \circ (\tau^\bullet \otimes \Id) \circ \lambda^\bullet \circ \left[\left(m_{s+1} \circ (I_{j_1} \otimes \cdots \otimes I_{j_s} \otimes \Id) \circ \psi^\bullet \right) \otimes \Id \right] \circ \sigma^\bullet \\
	\end{split}
\end{equation*}

\vspace{1em}
\noindent \textbf{Step 5.}
	Rewrite the sum by commuting composition and tensor product and considering the diagonal $\alpha = r+s$ instead of $r$ and $s$ individually.  Observe that one of $r$ and $s$ can be 0, but not both at the same time.

	\begin{equation*}
	\begin{split}
		\sum_{\substack{i+j=n+1}} &\sum_{\sigma(i)=n} \sum_{\substack{1 \leq \alpha \leq n-1 \\ r+s=\alpha \\r,s\geq 0}} \sum_{\substack{\tau \in S'(i_1,\ldots, i_r) \\ i_1 +\ldots + i_r = j-1}}  \sum_{\substack{\psi\in S'(j_1,\ldots,j_s) \\ j_1+\ldots+j_s=i-1}} \\
		&f_{r+1} \circ \omega^\bullet \circ (m_{s+1} \otimes \Id) \circ (I_{j_1} \otimes \cdots \otimes I_{j_s} \otimes \Id \otimes I_{i_1} \otimes \cdots \otimes I_{i_r}) \circ (\psi^\bullet \otimes \Id \otimes \tau^\bullet) \circ \sigma^\bullet
	\end{split}
	\end{equation*}

\vspace{1em}
\noindent \textbf{Step 6.}
	Apply Lemma \ref{UL} to obtain

	\begin{align*}
		&\sum_{\substack{\pi \in S'(c_1,\ldots, c_\alpha) \\ c_1 +\ldots + c_\alpha = n-1}} \sum_{\substack{\theta \in S(t, \alpha+1-t) \\ \theta(t) = \alpha+1 \\ 1 \leq t \leq \alpha+1}}  f_{\alpha+2-t} \circ \omega^\bullet \circ (m_{t} \otimes \Id) \circ \theta^\bullet \circ (I_{c_1} \otimes \cdots \otimes I_{c_\alpha}\otimes \Id) \circ (\pi^\bullet \otimes \Id)\\
	\end{align*}

\vspace{1em}
\noindent \textbf{Step 7.}
	Now, focusing on the $l$ terms (the second sum in Step 3), our goal is to apply the $L_\infty$-algebra relation for $I$.  The steps we follow here are essentially the same as in Lemma 2 (steps 3-12), and we direct the reader to them for details and for diagrams.  We start with
	$$
		\sum_{\substack{i+j=n+1 \\ 1 \leq i < n}} \sum_{\sigma(n)=n} \sum_{r=1}^{j-1} \sum_{\substack{\tau \in S'(i_1,\ldots, i_r) \\ i_1 +\ldots + i_r = j-1}} f_{r+1} \circ (I_{i_1} \otimes \cdots \otimes I_{i_r}\otimes \Id) \circ (\tau^\bullet \otimes \Id) \circ (l'_i \otimes \Id) \circ \sigma^\bullet
	$$
	Denote the block where $l_i'$ goes by $I_{i_l}$.  Break down the sum by $i_l = s$.
	$$
		\sum_{\substack{i+j=n+1 \\ 1 \leq i < n}} \sum_{\sigma(n)=n} \sum_{r=1}^{j-1} \sum_{s=1}^{j-1} \sum_{\substack{\tau \in S'(i_1,\ldots, i_r) \\ i_1 +\ldots + i_r = j-1 \\ i_l = s}} f_{r+1} \circ (I_{i_1} \otimes \cdots \otimes I_{i_r}\otimes \Id) \circ (\tau^\bullet \otimes \Id) \circ (l'_i \otimes \Id) \circ \sigma^\bullet
	$$
	Remove $j$ from the notation.
	$$
		\sum_{i=1}^{n-1} \sum_{\sigma(n)=n} \sum_{r=1}^{n-i} \sum_{s=1}^{n-i} \sum_{\substack{\tau \in S'(i_1,\ldots, i_r) \\ i_1 +\ldots + i_r = n-i \\ i_l = s}} f_{r+1} \circ (I_{i_1} \otimes \cdots \otimes I_{i_r}\otimes \Id) \circ (\tau^\bullet \otimes \Id) \circ (l'_i \otimes \Id) \circ \sigma^\bullet
	$$
	Reindex over the sum of $i$ and $s$.
	$$
		\sum_{k=1}^{n-1} \sum_{i+s=k+1} \sum_{\sigma(n)=n} \sum_{\substack{\tau \in S'(i_1,\ldots, i_r) \\ i_1 +\ldots + i_r = n-i \\ i_l = s}} f_{r+1} \circ (I_{i_1} \otimes \cdots \otimes I_{i_r}\otimes \Id) \circ (\tau^\bullet \otimes \Id) \circ (l'_i \otimes \Id) \circ \sigma^\bullet
	$$
	Use the map $\lambda^\bullet$ to permute $l_i'$ around $\tau$ and change $\tau$ to $\tau'$.
	\begin{equation*}
		\begin{split}
			\sum_{k=1}^{n-1} \sum_{i+s=k+1} & \sum_{\sigma(n)=n} \sum_{\substack{\tau' \in S(i_1,\ldots, i_l-1, \ldots, i_r) \\ i_1 +\ldots + i_r = n-i \\ i_l = s}} \\
			&f_{r+1} \circ (I_{i_1} \otimes \cdots \otimes I_{i_l} \otimes \cdots \otimes I_{i_r}\otimes \Id) \circ \lambda^\bullet \circ (\Id \otimes \tau'^\bullet \otimes \Id) \circ (l'_i \otimes \Id) \circ \sigma^\bullet
		\end{split}
	\end{equation*}
	Combine $\tau'$ and $\sigma$ into the permutation $\eta$.
	\begin{equation*}
		\begin{split}
			\sum_{k=1}^{n-1} \sum_{i+s=k+1} &\sum_{\rho\in S(i,i_l-1)} \sum_{\substack{\eta \in S(i+i_l-1, i_1,\ldots, \widehat{i_l}, \ldots, i_r,1) \\ i_1 +\ldots + i_r = n-i \\ i_l = s \\ \eta(n) = n}} \\
			&f_{r+1} \circ (I_{i_1} \otimes \cdots \otimes I_{i_l} \otimes \cdots \otimes I_{i_r}\otimes \Id) \circ \lambda^\bullet \circ \omega^\bullet \circ (l'_i \otimes \Id) \circ (\rho^\bullet \otimes \Id) \circ \eta^\bullet
		\end{split}
	\end{equation*}
	Use skew-symmetry of $f_{r+1}$ to swap the order of the $I$'s
	$$
		\sum_{k=1}^{n-1} \sum_{i+s=k+1} \sum_{\rho\in S(i,i_l-1)} \sum_{\substack{\eta \in S(i+i_l-1, i_1,\ldots, \widehat{i_l}, \ldots, i_r,1) \\ i_1 +\ldots + i_r = n-i \\ i_l = s \\ \eta(n) = n}} f_{r+1} \circ (I_{i_l} \otimes I_{i_1} \otimes \cdots \otimes I_{i_r}\otimes \Id) \circ (l'_i \otimes \Id) \circ (\rho^\bullet \otimes \Id) \circ \eta^\bullet
	$$
	Rewrite suggestively, noting that now the $I_{i_l}$ is omitted from $I_{i_1} \otimes \cdots \otimes I_{i_r}$.
	\begin{equation*}
		\begin{split}
			\sum_{k=1}^{n-1} \sum_{i+s=k+1} \sum_{\rho\in S(i,i_l-1)} \sum_{\substack{\eta \in S(k, i_1,\ldots, \widehat{i_l}, \ldots, i_r,1) \\ i_1 +\ldots + i_r = n-i \\ i_l = s \\ \eta(n) = n}} f_{r+1} \bigg (\left( I_{i_l} \circ (l_i' \otimes \Id) \circ \rho^\bullet \right) \otimes \left( I_{i_1} \otimes \cdots \otimes I_{i_r} \right) \otimes \Id \bigg) \circ \eta^\bullet
		\end{split}
	\end{equation*}
	Apply the morphism relations.  

	$$
		\sum_{k=1}^{n-1} \sum_{\substack{\gamma \in S'(t_1, \ldots, t_z) \\ t_1 + \ldots + t_z = k \\ 1\leq z \leq k}} \sum_{\substack{\eta \in S(k, i_1,\ldots, \widehat{i_l}, \ldots, i_r,1) \\ i_1 +\ldots + \widehat{i_l} + \ldots + i_r = n-1-k \\ \eta(n) = n}} f_{r+1} \bigg( \left( l_z \circ (I_{t_1} \otimes \cdots \otimes I_{t_z}) \circ \gamma^\bullet \right) \otimes \left( I_{i_1} \otimes \cdots \otimes I_{i_r} \right) \otimes \Id \bigg) \circ \eta^\bullet
	$$
	Combine $\gamma$ and $\eta$ into $\psi$.  
	\begin{equation*}
		\begin{split}
		\sum_{k=1}^{n-1} \sum_{\substack{\psi \in S(t_1, \ldots, t_z, i_1,\ldots, \widehat{i_l}, \ldots, i_r,1) \\ 1 \leq z \leq k \\ t_1 + \ldots + t_z = k \\ i_1 +\ldots + \widehat{i_l} + \ldots + i_r = n-1-k \\ \psi(n) = n}} f_{r+1} \bigg( (l_z\otimes \Id) \circ (I_{t_1} \otimes \cdots \otimes I_{t_z} \otimes I_{i_1} \otimes \cdots \otimes I_{i_r} \otimes \Id) \circ \psi^\bullet\bigg)
		\end{split}
	\end{equation*}
	This is equivalent to
	$$
		\sum_{\substack{\pi \in S'(c_1, \ldots, c_\alpha) \\ c_1 + \ldots + c_\alpha = n-1 \\ 1\leq \alpha \leq n-1}} \sum_{\substack{\theta \in S(t, \alpha+1-t) \\ \theta(\alpha+1) = \alpha+1 \\ 1\leq t < \alpha + 1}} f_{\alpha+2-t} \circ (l_t \otimes \Id) \circ \theta^\bullet \circ (I_{c_1} \otimes \cdots \otimes I_{c_\alpha} \otimes \Id) \circ (\pi^\bullet \otimes \Id)
	$$

\vspace{1em}
\noindent \textbf{Step 8.}
	In total, combining this with Step 6, we have the sum 

	\begin{align*}
		&\sum_{\substack{\pi \in S'(c_1,\ldots, c_\alpha) \\ c_1 +\ldots + c_\alpha = n-1}} \sum_{\substack{\theta \in S(t, \alpha+1-t) \\ \theta(t) = \alpha+1 \\ 1 \leq t \leq \alpha+1}}  f_{\alpha+2-t} \circ \omega^\bullet \circ (m_{t} \otimes \Id) \circ \theta^\bullet \circ (I_{c_1} \otimes \cdots \otimes I_{c_\alpha}\otimes \Id) \circ (\pi^\bullet \otimes \Id)\\
		+& \sum_{\substack{\pi \in S'(c_1, \ldots, c_\alpha) \\ c_1 + \ldots + c_\alpha = n-1}} \sum_{\substack{\theta \in S(t, \alpha+1-t) \\ \theta(\alpha+1) = \alpha+1 \\ 1\leq t < \alpha + 1}} f_{\alpha+2-t} \circ (l_t \otimes \Id) \circ \theta^\bullet \circ (I_{c_1} \otimes \cdots \otimes I_{c_\alpha} \otimes \Id) \circ (\pi^\bullet \otimes \Id)
	\end{align*}

\vspace{1em}
\noindent \textbf{Step 9.}
	Change notation; change $t$ to $i$ and $\alpha+2-t$ to $j$.

	\begin{align*}
		&\sum_{\substack{\pi \in S'(c_1,\ldots, c_\alpha) \\ c_1 +\ldots + c_\alpha = n-1}} \sum_{i+j=\alpha+2} \sum_{\substack{\theta \in S(i, \alpha+1-i) \\ \theta(i) = \alpha+1}}  f_{j} \circ \omega^\bullet \circ (m_{i} \otimes \Id) \circ \theta^\bullet \circ (I_{c_1} \otimes \cdots \otimes I_{c_\alpha}\otimes \Id) \circ (\pi^\bullet \otimes \Id)\\
		+& \sum_{\substack{\pi \in S'(c_1, \ldots, c_\alpha) \\ c_1 + \ldots + c_\alpha = n-1}} \sum_{\substack{i+j=\alpha+2 \\ i<\alpha+1}} \sum_{\substack{\theta \in S(i, \alpha+1-i) \\ \theta(\alpha+1) = \alpha+1}} f_{j} \circ (l_i \otimes \Id) \circ \theta^\bullet \circ (I_{c_1} \otimes \cdots \otimes I_{c_\alpha} \otimes \Id) \circ (\pi^\bullet \otimes \Id)
	\end{align*}

\vspace{1em}
\noindent \textbf{Step 10.}  Applying the module homomorphism relation for $f$, we obtain
\begin{align*}
		&\sum_{\substack{\pi \in S'(c_1,\ldots, c_\alpha) \\ c_1 +\ldots + c_\alpha = n-1}} \sum_{r+s=\alpha+2} \sum_{\substack{\rho \in S(\alpha-s, s)}}  n_{r} \circ (\Id \otimes f_s) \circ (\rho^\bullet \otimes \Id) \circ (I_{c_1} \otimes \cdots \otimes I_{c_\alpha}\otimes \Id) \circ (\pi^\bullet \otimes \Id)
	\end{align*}

\vspace{1em}
\noindent \textbf{Step 11.} It just remains to show that the sum above is equal to $$\sum_{r+s=n+1} \sum_{\tau} n'_r \circ (\Id \otimes (I^*f)_s) \circ (\tau^\bullet \otimes \Id) $$
Therefore, use the definition of $I^*f$.  Like usual, we start indexing at $x=0$ to allow for the $f_1$ case.

\begin{align*}
	\sum_{r+s=n+1} \sum_{\tau\in S(n-s,s)} \sum_{\substack{\phi \in S'(i_1,\ldots, i_x) \\ i_1 +\ldots + i_x = s-1 \\ 0 \leq x \leq s-1}} n'_r \circ \Big[\Id \otimes \Big(f_{x+1} \circ [(I_{i_1} \otimes \cdots \otimes I_{i_x} \circ\phi^\bullet) \otimes \Id] \Big) \Big]  \circ (\tau^\bullet \otimes \Id)
\end{align*}

\vspace{1em}
\noindent \textbf{Step 12.} Now use the definition of $n'$.  Allow for $y=0$ to deal with the $n_1$ case.

\begin{equation*}
	\begin{split}
		\sum_{r+s=n+1} & \sum_{\tau\in S(n-s,s)} \sum_{x=0}^{s-1} \sum_{\substack{\phi \in S'(i_1,\ldots, i_x) \\ i_1 +\ldots + i_x = s-1}} \sum_{y=0}^{n-s} \sum_{\substack{\gamma \in S'(j_1,\ldots, j_y) \\ j_1 +\ldots + j_y = n-s-1}}  n_{y+1} \circ (I_{j_1} \otimes \cdots \otimes I_{j_y} \otimes \Id) \circ \\ & (\gamma^\bullet \otimes \Id) \circ  \Big[\Id \otimes \Big(f_{x+1} \circ (I_{i_1} \otimes \cdots \otimes I_{i_x} \otimes \Id) \circ (\phi^\bullet\otimes \Id) \Big) \Big] \circ (\tau^\bullet \otimes \Id)
	\end{split}
\end{equation*}

\vspace{1em}
\noindent \textbf{Step 13.} Commute composition and tensor product to rewrite as 

\begin{equation*}
	\begin{split}
		&\sum_{r+s=n+1} \sum_{\tau\in S(n-s,s)} \sum_{x=0}^{s-1} \sum_{\substack{\phi \in S'(i_1,\ldots, i_x) \\ i_1 +\ldots + i_x = s-1}} \sum_{y=0}^{n-s} \sum_{\substack{\gamma \in S'(j_1,\ldots, j_y) \\ j_1 +\ldots + j_y = n-s-1}}  \\
		&\qquad n_{y+1}  \circ (\Id \otimes f_{x+1}) \circ (I_{j_1} \otimes \cdots \otimes I_{j_y} \otimes I_{i_1} \otimes \cdots \otimes I_{i_x} \otimes \Id) \circ (\gamma^\bullet \otimes \phi^\bullet \otimes \Id) \circ (\tau^\bullet \otimes \Id)
	\end{split}
\end{equation*}

\vspace{1em}
\noindent \textbf{Step 14.} Reindex over the diagonal of $\alpha=x+y$.  Observe that one of $x$ and $y$ can be 0, but not both at the same time.

\begin{equation*}
	\begin{split}
		&\sum_{r+s=n+1} \sum_{\tau\in S(n-s,s)} \sum_{\substack{1 \leq \alpha \leq n-1 \\ x+y=\alpha\\x,y\geq 0}} \sum_{\substack{\phi \in S'(i_1,\ldots, i_x) \\ i_1 +\ldots + i_x = s-1}} \sum_{\substack{\gamma \in S'(j_1,\ldots, j_y) \\ j_1 +\ldots + j_y = n-s-1}}  \\
		&\qquad n_{y+1}  \circ (\Id \otimes f_{x+1}) \circ (I_{j_1} \otimes \cdots \otimes I_{j_y} \otimes I_{i_1} \otimes \cdots \otimes I_{i_x} \otimes \Id) \circ (\gamma^\bullet \otimes \phi^\bullet \otimes \Id) \circ (\tau^\bullet \otimes \Id)
	\end{split}
\end{equation*}

\vspace{1em}
\noindent \textbf{Step 15.} Apply Lemma \ref{UL} to get

\begin{equation*}
	\begin{split}
		&\sum_{\substack{\pi \in S'(c_1,\ldots, c_\alpha) \\ c_1 +\ldots + c_\alpha = n-1 \\ 1 \leq \alpha \leq n-1}}  \sum_{\substack{\theta \in S(\alpha-s,s) \\  1\leq s \leq \alpha}} n_{\alpha+2-s}  \circ (\Id \otimes f_{s}) \circ \theta^\bullet \circ (I_{c_1} \otimes \cdots \otimes I_{c_\alpha}\otimes \Id) \circ (\pi^\bullet \otimes \Id)
	\end{split}
\end{equation*}

\vspace{1em}
\noindent \textbf{Step 16.} Rewriting this as 

\begin{equation*}
	\begin{split}
		&\sum_{\substack{\pi \in S'(c_1,\ldots, c_\alpha) \\ c_1 +\ldots + c_\alpha = n-1}} \sum_{r+s=\alpha+2} \sum_{\substack{\theta \in S(\alpha-s,s) \\  1\leq s \leq \alpha}} n_{r}  \circ (\Id \otimes f_{s}) \circ \theta^\bullet \circ (I_{c_1} \otimes \cdots \otimes I_{c_\alpha}\otimes \Id) \circ (\pi^\bullet \otimes \Id)
	\end{split}
\end{equation*}
shows that it is the same as the sum in Step 10, which completes the proof.

\end{proof}

\begin{theorem}[Functoriality]
	Suppose $I: (L',l') \to (L,l)$ is a map of $L_\infty$-algebras.  Then $I^*: L\Mod \to L'\Mod$ is a functor.
\end{theorem}
\begin{proof}
	Suppose we have $L_\infty$-modules $M,N,$ and $Q$ over $L$ and $L_\infty$-module homomorphisms $M \xrightarrow{f} N \xrightarrow{g} Q$.  We have defined $I^*$ on objects and morphisms, so it remains to show that $I^*(\Id_M) = \Id_{I^*M}$ and that $I^*(g\circ f) = I^*g \circ I^*f$.  For the former, observe that $(I^*(\Id_M))_1=(\Id_M)_1$, and for $n\geq 2$,
	$$(I^*(\Id_M))_n = \sum_{r=1}^{n-1} \sum_{\substack{\tau \in S'(i_1,\ldots, i_r) \\ i_1 +\ldots + i_r = n-1}} (\Id_M)_{r+1} \circ (I_{i_1} \otimes \cdots \otimes I_{i_r} \otimes \Id) \circ (\tau^{\bullet} \otimes \Id)$$
	But $(\Id_M)_r=0$ for $r>1$, and so we conclude that $(I^*(\Id_M))_n=0$ for $n\geq 2$.  Hence $I^*(\Id_M) = \Id_{I^*M}$.  

	In remains to show that $I^*(g\circ f) = I^*g \circ I^*f$.  We will follow essentially the same procedure as in Lemma 2, steps 13-17.

	\vspace{1em}
	\noindent \textbf{Step 1.} We start with the right-hand side, and replace $[I^*g \circ I^*f]_n$ with its definition

	\begin{equation*}
		\begin{split}
			\sum_{i+j=n+1} \sum_{\sigma(i)=n} (I^*g)_j \circ \lambda^{\bullet} \circ ((I^*f)_i \otimes \Id) \circ \sigma^{\bullet}
		\end{split}
	\end{equation*}

	\vspace{1em}
	\noindent \textbf{Step 2.} Replace $I^*g$ and $I^*f$ with their definitions.  
	\begin{equation*}
		\begin{split}
			\sum_{i+j=n+1} & \sum_{\sigma(i)=n} \sum_{r=0}^{i-1} \sum_{\substack{\phi \in S'(i_1,\ldots, i_r) \\ i_1 +\ldots + i_r = i-1}} \sum_{s=0}^{j-1} \sum_{\substack{\psi \in S'(j_1,\ldots, j_s) \\ j_1 +\ldots + j_s = j-1}} \big[g_{s+1} \circ (I_{j_1} \otimes \cdots \otimes I_{j_s} \otimes \Id) \circ (\psi^{\bullet}\otimes \Id) \big]  \\ & \circ \lambda^{\bullet} \circ (\big[ f_{r+1} \circ (I_{i_1} \otimes \cdots \otimes I_{i_r} \otimes \Id) \circ (\phi^{\bullet} \otimes \Id) \big] \otimes \Id) \circ \sigma^{\bullet}
		\end{split}
	\end{equation*}
	Note that we include the cases $r=0$ and $s=0$ to include the cases $f_1$ and $g_1$, respectively.  In particular, $r=0$ will contribute a nonzero term only when $i=1$, and $s=0$ will only contribute a nonzero term when $j=1$.

	\vspace{1em}
	\noindent \textbf{Step 3.} Commute composition and tensor product to rewrite. 
	\begin{equation*}
		\begin{split}
			\sum_{i+j=n+1} & \sum_{\sigma(i)=n} \sum_{r=0}^{i-1} \sum_{\substack{\phi \in S'(i_1,\ldots, i_r) \\ i_1 +\ldots + i_r = i-1}} \sum_{s=0}^{j-1} \sum_{\substack{\psi \in S'(j_1,\ldots, j_s) \\ j_1 +\ldots + j_s = j-1}}  \\ 
			&g_{s+1} \circ \lambda'^{\bullet} \circ (f_{r+1} \otimes \Id) \circ 
			(I_{i_1}\otimes \cdots I_{i_r}  \otimes \Id \otimes I_{j_1}\otimes \cdots \otimes I_{j_s}) \circ (\phi^{\bullet} \otimes \Id \otimes \psi^{\bullet}) \circ \sigma^{\bullet}
		\end{split}
	\end{equation*}
	Here, $\lambda'^{\bullet}$ is the map that permutes the module element into the last input of $g_{s+1}$.

	\vspace{1em}
	\noindent \textbf{Step 4.} By Lemma 4, we obtain

	\begin{equation*}
		\begin{split}
			\sum_{\substack{\tau\in S'(c_1,\ldots,c_\alpha) \\ c_1+\ldots+c_\alpha=n-1 \\ 1\leq \alpha \leq n-1}} \sum_{\substack{\theta \in S(t+1,\alpha-t) \\ \theta(t+1)=\alpha+1 \\  0 \leq t \leq \alpha}}  g_{\alpha+1-t} \circ \lambda'^{\bullet} \circ (f_{t+1}\otimes \Id) \circ \theta^{\bullet} \circ (I_{c_1} \otimes \cdots \otimes I_{c_\alpha}\otimes \Id) \circ (\tau^{\bullet} \otimes \Id)
		\end{split}
	\end{equation*}

	\vspace{1em}
	\noindent \textbf{Step 5.} Change notation; let $p=t+1$ and $q=\alpha+1-t$.

	\begin{equation*}
		\begin{split}
			\sum_{\substack{\tau\in S'(c_1,\ldots,c_\alpha) \\ c_1+\ldots+c_\alpha=n-1 \\ 1\leq \alpha \leq n-1}} \sum_{\substack{\theta \in S(p,\alpha+1-p) \\ \theta(p)=\alpha+1 \\  1 \leq p \leq \alpha+1}}  g_{q} \circ \lambda'^{\bullet} \circ (f_{p}\otimes \Id) \circ \theta^{\bullet} \circ (I_{c_1} \otimes \cdots \otimes I_{c_\alpha}\otimes \Id) \circ (\tau^{\bullet} \otimes \Id)
		\end{split}
	\end{equation*}

	\vspace{1em}
	\noindent \textbf{Step 6.} By the definition of $g\circ f$, this is

	\begin{equation*}
		\begin{split}
			\sum_{\substack{\tau\in S'(c_1,\ldots,c_\alpha) \\ c_1+\ldots+c_\alpha=n-1 \\ 1\leq \alpha \leq n-1}}  (g\circ f)_{\alpha+1} \circ (I_{c_1} \otimes \cdots \otimes I_{c_\alpha}\otimes \Id) \circ (\tau^{\bullet} \otimes \Id)
		\end{split}
	\end{equation*}

	\vspace{1em}
	\noindent \textbf{Step 7.} By the definition of $I^*$, this is precisely $[I^*(g\circ f)]_n$, as desired.

\end{proof}

\begin{corollary}
	If $L$ and $L'$ are Lie algebras, and $\phi: L'\to L$ is a Lie algebra homomorphism, $\phi^*$ is the usual restriction of scalars for Lie algebra representations.
\end{corollary}
\begin{proof}
	Let $\rho: L \to \gl(M)$ be a Lie algebra representation.  For $x \in L'$ and $m\in M$, the usual restriction of scalars for Lie algebra representations is given by $x\cdot m := \phi(x)\cdot m$.  Indeed, $\rho': L' \to \gl(M)$ defined by $\rho'(y) = \rho(\phi(y))$ is a homomorphism of Lie algebras.  Now, regarding $\phi$ as an $L_\infty$-algebra map with $\phi_i=0$ for $i\neq 1$, because there are also no higher operations on $M$ as an $L_\infty$ $L$-module, the formulas given in Lemma \ref{Lem:Objects} for the induced operation simplify to give the usual restriction of scalars operation described above.
\end{proof}

We now prove the technical lemma that was used in the main results above.  In particular, this lemma gives two ways to interpret a particular composition of unshuffles.

\begin{lemma}\label{UL}
For a fixed $n$,

	$$\sum_{p=1}^n \sum_{\sigma(p)=n} \sum_{\substack{1 \leq \alpha \leq n-1 \\ r+s=\alpha \\ r,s \geq 0}} \sum_{\substack{\phi\in S'(i_1,\ldots,i_r) \\ i_1+\ldots+i_r=p-1}} \sum_{\substack{\psi\in S'(j_1,\ldots,j_s) \\ j_1+\ldots+j_s=n-p}} (I_{i_1}\otimes \cdots I_{i_r} \otimes \Id \otimes I_{j_1}\otimes \cdots \otimes I_{j_s}) \circ (\phi^{\bullet} \otimes \Id \otimes \psi^{\bullet}) \circ \sigma^{\bullet}$$
	is the same as
	$$
	\sum_{\substack{\tau\in S'(c_1,\ldots,c_\alpha) \\ c_1+\ldots+c_\alpha=n-1}} \sum_{\substack{\theta \in S(r+1,\alpha-r) \\ \theta(r+1)=\alpha+1 \\  0 \leq r \leq \alpha}}  \theta^{\bullet} \circ (I_{c_1} \otimes \cdots \otimes I_{c_\alpha}\otimes \Id) \circ (\tau^{\bullet} \otimes \Id).
	$$
 \end{lemma}
 \begin{proof}

  \renewcommand{\figurename}{Figure}
	 \renewcommand{\thefigure}{4}
	\begin{figure}[H]
		\centering
		\scalebox{.96}{\input{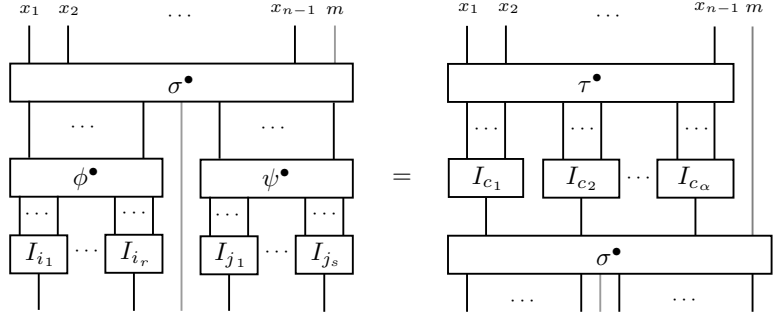}}
    	\caption{The left-hand side represents first unshuffling $n$ elements into two boxes (with the module element by itself) via $\sigma^\bullet$ and then unshuffling these boxes further into $r$ boxes and $s$ boxes via $\phi^\bullet$ and $\psi^\bullet$, respectively.  The right-hand side represents first unshuffling $n-1$ elements into $\alpha$ boxes via $\tau^\bullet$ and then unshuffling these $\alpha$ boxes via $\sigma^\bullet$.} 
    	\label{fig:TL}
	\end{figure}


 To see this, it is helpful to examine what the first sum does for a fixed $p$ and a fixed $\alpha$.  It unshuffles $n$ elements into a box of size $p-1$ and a box of size $n-p$, with the module element in between.  It then unshuffles the box of size $p-1$ further via $\phi$ into $r$ smaller boxes and the box of size $n-p$ further via $\psi$ into $s$ smaller boxes.  

 So, if we iterate through $\alpha=r+s$, this sum describes all possible ways of unshuffling $n$ elements into $r$ boxes (which contain a total of $p-1$ elements) and $s$ boxes (which contain a total number of $n-p$ elements), with the module element in between.  Then, iterating through all possible $p$ tells us that the sum describes all ways of unshuffling $n$ elements into $r+s$ boxes, with the module element in between.  Note that the $r$ boxes and the $s$ boxes have to be of increasing size when considered separately, but they need not be in order when considered all together (e.g. some of the $s$ boxes could be smaller than the last $r$ box).

On the other hand, the second sum unshuffles the $n-1$ algebra elements into $\alpha$ boxes first (here, the boxes are all of increasing size), and then it picks out $r$ of these via an $r$-unshuffle $\theta$ in $S_\alpha$.  Since there was a module element between the $r$ boxes and $s$ boxes in the first sum, we can view $\theta$ as an $(r+1)$-unshuffle in $S_\alpha$ where it puts the module element after the $r$ boxes.  So what we have done is the same as before: unshuffle $n$ elements into a group of $r$ boxes, a module element, and a group of $s=\alpha-r$ boxes, where the boxes are of increasing order when considered separately (but not necessarily when considered all together), see Figure \ref{fig:TL}.  An explicit correspondence between the two sums can be written down using formulas.
\end{proof}

\section{Appendix} 

\subsection{Composition}\

\begin{figure}[H]
\begin{minipage}[c]{1\textwidth}
	\centering
	\scalebox{1}{\input{figs/Module/Step-1.tex}}
\end{minipage}
\hfill
\begin{minipage}[c]{0\textwidth}
	\centering	
	\
\end{minipage}
\begin{minipage}[t]{1\textwidth}
	\centering
    \caption*{Step 1}
\end{minipage}
    \hfill
\begin{minipage}[t]{0\textwidth}
	\centering
    \
\end{minipage}
\end{figure}
\begin{figure}[H]
\begin{minipage}[c]{1\textwidth}
	\centering
	\scalebox{1}{\input{figs/Module/Step-2.tex}}
\end{minipage}
\hfill
\begin{minipage}[c]{0\textwidth}
	\centering	
	\
\end{minipage}
\begin{minipage}[t]{1\textwidth}
	\centering
    \caption*{Step 2}
\end{minipage}
    \hfill
\begin{minipage}[t]{0\textwidth}
	\centering
    \
\end{minipage}
\end{figure}
\
\vfill

\begin{figure}[H]
\begin{minipage}[c]{1\textwidth}
	\centering
	\scalebox{.95}{\input{figs/Module/Step-3.tex}}
\end{minipage}
\hfill
\begin{minipage}[c]{0\textwidth}
	\centering	
	\
\end{minipage}
\begin{minipage}[t]{1\textwidth}
	\centering
    \caption*{Step 3}
\end{minipage}
    \hfill
\begin{minipage}[t]{0\textwidth}
	\centering
    \
\end{minipage}
\end{figure}
\begin{figure}[H]
\begin{minipage}[c]{1\textwidth}
	\centering
	\scalebox{1}{\input{figs/Module/Step-5.tex}}
\end{minipage}
\hfill
\begin{minipage}[c]{0\textwidth}
	\centering	
	\
\end{minipage}
\begin{minipage}[t]{1\textwidth}
	\centering
    \caption*{Step 5}
\end{minipage}
    \hfill
\begin{minipage}[t]{0\textwidth}
	\centering
    \
\end{minipage}
\end{figure}
\
\vfill

\begin{figure}[H]
\begin{minipage}[c]{1\textwidth}
	\centering
	\scalebox{.95}{\input{figs/Module/Step-6.tex}}
\end{minipage}
\hfill
\begin{minipage}[c]{0\textwidth}
	\centering	
	\
\end{minipage}
\begin{minipage}[t]{1\textwidth}
	\centering
    \caption*{Step 6}
\end{minipage}
    \hfill
\begin{minipage}[t]{0\textwidth}
	\centering
    \
\end{minipage}
\end{figure}
\begin{figure}[H]
\begin{minipage}[c]{1\textwidth}
	\centering
	\scalebox{1}{\input{figs/Module/Step-8.tex}}
\end{minipage}
\hfill
\begin{minipage}[c]{0\textwidth}
	\centering	
	\
\end{minipage}
\begin{minipage}[t]{1\textwidth}
	\centering
    \caption*{Step 8}
\end{minipage}
    \hfill
\begin{minipage}[t]{0\textwidth}
	\centering
    \
\end{minipage}
\end{figure}
\
\vfill

\begin{figure}[H]
\begin{minipage}[c]{.45\textwidth}
	\centering
	\scalebox{1}{\input{figs/Module/Step-9.tex}}
\end{minipage}
\hfill
\begin{minipage}[c]{.45\textwidth}
	\centering	
	\scalebox{1}{\input{figs/Module/Step-10.tex}}
\end{minipage}
\begin{minipage}[t]{0.45\textwidth}
	\centering
    \caption*{Step 9}
\end{minipage}
    \hfill
\begin{minipage}[t]{0.45\textwidth}
	\centering
    \caption*{Step 10}
\end{minipage}
\end{figure}
\begin{figure}[H]
\begin{minipage}[c]{.45\textwidth}
	\centering
	\scalebox{1}{\input{figs/Module/Step-11.tex}}
\end{minipage}
\hfill
\begin{minipage}[c]{.45\textwidth}
	\centering	
	\scalebox{1}{\input{figs/Module/Step-13.tex}}
\end{minipage}
\begin{minipage}[t]{0.45\textwidth}
	\centering
    \caption*{Step 11}
\end{minipage}
    \hfill
\begin{minipage}[t]{0.45\textwidth}
	\centering
    \caption*{Step 13}
\end{minipage}
\end{figure}
\
\vfill

\newpage
\subsection{Objects} \

\begin{figure}[H]
\begin{minipage}[c]{1\textwidth}
	\centering
	\scalebox{1}{\input{figs/Functor/Objects/Step-1.tex}}
\end{minipage}
\hfill
\begin{minipage}[c]{0\textwidth}
	\centering	
	\
\end{minipage}
\begin{minipage}[t]{1\textwidth}
	\centering
    \caption*{Step 1}
\end{minipage}
    \hfill
\begin{minipage}[t]{0\textwidth}
	\
\end{minipage}
\end{figure}
\begin{figure}[H]
\begin{minipage}[c]{.45\textwidth}
	\centering
	\scalebox{1}{\input{figs/Functor/Objects/Step-2.tex}}
\end{minipage}
\hfill
\begin{minipage}[c]{.45\textwidth}
	\centering	
	\scalebox{1}{\input{figs/Functor/Objects/Step-4.tex}}
\end{minipage}
\begin{minipage}[t]{0.45\textwidth}
	\centering
    \caption*{Step 2}
\end{minipage}
    \hfill
\begin{minipage}[t]{0.45\textwidth}
	\centering
    \caption*{Step 4}
\end{minipage}
\end{figure}
\
\vfill

\begin{figure}[H]
\begin{minipage}[c]{.35\textwidth}
	\centering
	\scalebox{1}{\input{figs/Functor/Objects/Step-5.tex}}
\end{minipage}
\hfill
\begin{minipage}[c]{.55\textwidth}
	\centering	
	\scalebox{.8}{\input{figs/Functor/Objects/Step-6.tex}}
\end{minipage}
\begin{minipage}[t]{0.35\textwidth}
	\centering
    \caption*{Step 5}
\end{minipage}
    \hfill
\begin{minipage}[t]{0.55\textwidth}
	\centering
    \caption*{Step 6}
\end{minipage}
\end{figure}
\begin{figure}[H]
\begin{minipage}[c]{1\textwidth}
	\centering
	\scalebox{.85}{\input{figs/Functor/Objects/Step-7.tex}}
\end{minipage}
\hfill
\begin{minipage}[c]{0\textwidth}
	\centering	
	\
\end{minipage}
\begin{minipage}[t]{1\textwidth}
	\centering
    \caption*{Step 7}
\end{minipage}
    \hfill
\begin{minipage}[t]{0\textwidth}
	\centering
	\
\end{minipage}
\end{figure}
\
\vfill

\begin{figure}[H]
\begin{minipage}[c]{.45\textwidth}
	\centering
	\scalebox{.9}{\input{figs/Functor/Objects/Step-10.tex}}
\end{minipage}
\hfill
\begin{minipage}[c]{.45\textwidth}
	\centering	
	\scalebox{.9}{\input{figs/Functor/Objects/Step-11.tex}}
\end{minipage}
\begin{minipage}[t]{0.45\textwidth}
	\centering
    \caption*{Step 10}
\end{minipage}
    \hfill
\begin{minipage}[t]{0.45\textwidth}
	\centering
    \caption*{Step 11}
\end{minipage}
\end{figure}
\begin{figure}[H]
\begin{minipage}[c]{.45\textwidth}
	\centering
	\scalebox{1}{\input{figs/Functor/Objects/Step-12.tex}}
\end{minipage}
\hfill
\begin{minipage}[c]{.45\textwidth}
	\centering	
	\scalebox{1.1}{\input{figs/Functor/Objects/Step-13.tex}}
\end{minipage}
\begin{minipage}[t]{0.45\textwidth}
	\centering
    \caption*{Step 12}
\end{minipage}
    \hfill
\begin{minipage}[t]{0.45\textwidth}
	\centering
    \caption*{Step 13}
\end{minipage}
\end{figure}
\
\vfill

\begin{figure}[H]
\begin{minipage}[c]{.45\textwidth}
	\centering
	\scalebox{1}{\input{figs/Functor/Objects/Step-15.tex}}
\end{minipage}
\hfill
\begin{minipage}[c]{.45\textwidth}
	\centering	
	\scalebox{1.1}{\input{figs/Functor/Objects/Step-17.tex}}
\end{minipage}
\begin{minipage}[t]{0.45\textwidth}
	\centering
    \caption*{Step 15}
\end{minipage}
    \hfill
\begin{minipage}[t]{0.45\textwidth}
	\centering
    \caption*{Step 17}
\end{minipage}
\end{figure}
\begin{figure}[H]
\begin{minipage}[c]{1\textwidth}
	\centering
	\scalebox{1}{\input{figs/Functor/Objects/Step-19.tex}}
\end{minipage}
\hfill
\begin{minipage}[c]{0\textwidth}
	\centering	
	\
\end{minipage}
\begin{minipage}[t]{1\textwidth}
	\centering
    \caption*{Step 19}
\end{minipage}
    \hfill
\begin{minipage}[t]{0.45\textwidth}
	\centering
    \
\end{minipage}
\end{figure}
\
\vfill

\newpage

\newpage
\subsection{Morphisms} \

\begin{figure}[H]
\begin{minipage}[c]{1\textwidth}
	\centering
	\scalebox{1}{\input{figs/Functor/Morphisms/Step-1.tex}}
\end{minipage}
\hfill
\begin{minipage}[c]{0\textwidth}
	\centering	
	\
\end{minipage}
\begin{minipage}[t]{1\textwidth}
	\centering
    \caption*{Step 1}
\end{minipage}
    \hfill
\begin{minipage}[t]{0\textwidth}
	\centering
    \
\end{minipage}
\end{figure}
\begin{figure}[H]
\begin{minipage}[c]{1\textwidth}
	\centering
	\scalebox{1}{\input{figs/Functor/Morphisms/Step-3.tex}}
\end{minipage}
\hfill
\begin{minipage}[c]{0\textwidth}
	\centering	
	\
\end{minipage}
\begin{minipage}[t]{1\textwidth}
	\centering
    \caption*{Step 3}
\end{minipage}
    \hfill
\begin{minipage}[t]{0\textwidth}
	\centering
    \
\end{minipage}
\end{figure}
\
\vfill

\begin{figure}[H]
\begin{minipage}[c]{.45\textwidth}
	\centering
	\scalebox{.9}{\input{figs/Functor/Morphisms/Step-4.tex}}
\end{minipage}
\hfill
\begin{minipage}[c]{.45\textwidth}
	\centering	
	\scalebox{.9}{\input{figs/Functor/Morphisms/Step-5.tex}}
\end{minipage}
\begin{minipage}[t]{0.45\textwidth}
	\centering
    \caption*{Step 4}
\end{minipage}
    \hfill
\begin{minipage}[t]{0.45\textwidth}
	\centering
    \caption*{Step 5}
\end{minipage}
\end{figure}
\begin{figure}[H]
\begin{minipage}[c]{.45\textwidth}
	\centering
	\scalebox{1}{\input{figs/Functor/Morphisms/Step-6.tex}}
\end{minipage}
\hfill
\begin{minipage}[c]{.45\textwidth}
	\centering	
	\scalebox{1}{\input{figs/Functor/Morphisms/Step-7.tex}}
\end{minipage}
\begin{minipage}[t]{0.45\textwidth}
	\centering
    \caption*{Step 6}
\end{minipage}
    \hfill
\begin{minipage}[t]{0.45\textwidth}
	\centering
    \caption*{Step 7}
\end{minipage}
\end{figure}
\
\vfill

\begin{figure}[H]
\begin{minipage}[c]{1\textwidth}
	\centering
	\scalebox{1}{\input{figs/Functor/Morphisms/Step-9.tex}}
\end{minipage}
\hfill
\begin{minipage}[c]{0\textwidth}
	\centering	
	\
\end{minipage}
\begin{minipage}[t]{1\textwidth}
	\centering
    \caption*{Step 9}
\end{minipage}
    \hfill
\begin{minipage}[t]{0\textwidth}
	\centering
    \
\end{minipage}
\end{figure}
\begin{figure}[H]
\begin{minipage}[c]{.45\textwidth}
	\centering
	\scalebox{1}{\input{figs/Functor/Morphisms/Step-10.tex}}
\end{minipage}
\hfill
\begin{minipage}[c]{.45\textwidth}
	\centering	
	\scalebox{1.1}{\input{figs/Functor/Morphisms/Step-11.tex}}
\end{minipage}
\begin{minipage}[t]{0.45\textwidth}
	\centering
    \caption*{Step 10}
\end{minipage}
    \hfill
\begin{minipage}[t]{0.45\textwidth}
	\centering
    \caption*{Step 11}
\end{minipage}
\end{figure}
\
\vfill

\begin{figure}[H]
\begin{minipage}[c]{.45\textwidth}
	\centering
	\scalebox{.9}{\input{figs/Functor/Morphisms/Step-12.tex}}
\end{minipage}
\hfill
\begin{minipage}[c]{.45\textwidth}
	\centering	
	\scalebox{.9}{\input{figs/Functor/Morphisms/Step-13.tex}}
\end{minipage}
\begin{minipage}[t]{0.45\textwidth}
	\centering
    \caption*{Step 12}
\end{minipage}
    \hfill
\begin{minipage}[t]{0.45\textwidth}
	\centering
    \caption*{Step 13}
\end{minipage}
\end{figure}
\begin{figure}[H]
\begin{minipage}[c]{1\textwidth}
	\centering
	\scalebox{1}{\input{figs/Functor/Morphisms/Step-15.tex}}
\end{minipage}
\hfill
\begin{minipage}[c]{0\textwidth}
	\centering
	\
\end{minipage}
\begin{minipage}[t]{1\textwidth}
	\centering
    \caption*{Step 15}
\end{minipage}
    \hfill
\begin{minipage}[t]{0.45\textwidth}
	\centering
    \
\end{minipage}
\end{figure}
\
\vfill

\newpage
\subsection{Functoriality} \

\begin{figure}[H]
\begin{minipage}[c]{.45\textwidth}
	\centering
	\scalebox{1}{\input{figs/Functor/Functoriality/Step-1.tex}}
\end{minipage}
\hfill
\begin{minipage}[c]{.45\textwidth}
	\centering	
	\scalebox{.9}{\input{figs/Functor/Functoriality/Step-2.tex}}
\end{minipage}
\begin{minipage}[t]{0.45\textwidth}
    \caption*{Step 1}
\end{minipage}
    \hfill
\begin{minipage}[t]{0.45\textwidth}
    \caption*{Step 2}
\end{minipage}
\end{figure}
\begin{figure}[H]
\begin{minipage}[c]{.45\textwidth}
	\centering
	\scalebox{.95}{\input{figs/Functor/Functoriality/Step-3.tex}}
\end{minipage}
\hfill
\begin{minipage}[c]{.45\textwidth}
	\centering	
	\scalebox{.95}{\input{figs/Functor/Functoriality/Step-5.tex}}
\end{minipage}
\begin{minipage}[t]{0.45\textwidth}
    \caption*{Step 3}
\end{minipage}
    \hfill
\begin{minipage}[t]{0.45\textwidth}
    \caption*{Step 5}
\end{minipage}
\end{figure}
\
\vfill

\newpage

\begin{figure}[H]
\begin{minipage}[c]{.45\textwidth}
	\centering
	\scalebox{1}{\input{figs/Functor/Functoriality/Step-6.tex}}
\end{minipage}
\hfill
\begin{minipage}[c]{.45\textwidth}
	\centering	
	\
\end{minipage}
\begin{minipage}[t]{0.45\textwidth}
    \caption*{Step 6}
\end{minipage}
    \hfill
\begin{minipage}[t]{0.45\textwidth}
\
\end{minipage}
\end{figure}

\bibliographystyle{plain} 
\bibliography{infinity} 

\end{document}